\documentclass[11pt]{article}   
\usepackage{amssymb,amsmath,amsthm}
\usepackage{graphicx}
\usepackage{pgf,tikz}
\usepackage[english]{babel}
\usepackage{multicol}
\usepackage{enumerate, comment,multirow}

\newtheorem{theorem}{Theorem}
\newtheorem{corollary}[theorem]{Corollary}
\newtheorem{lemma}[theorem]{Lemma}

\newtheorem{conjecture}[theorem]{Conjecture}
\newtheorem{observation}[theorem]{Observation}

\newcommand{\note}[1]{{\color{blue} \sf Note: [#1]}}

\newcommand{\vanish}[1]{}

\begin{document}
\title{Odd Prime Graph Labelings} 

\author{
Holly Carter\\
{\small Austin Peay State University}\\
{\small hcarter12@my.apsu.edu}\\
\\
N. Bradley Fox\\
{\small Austin Peay State University}\\
{\small foxb@apsu.edu}\\
}

\date{}
\maketitle

\begin{abstract}
An odd prime labeling is a variation of a prime labeling in which the vertices of a graph of order~$n$ are labeled with the distinct odd integers $1$ to $2n-1$ so that the labels of adjacent vertices are relatively prime. This paper investigates many different classes of graphs including disjoint unions of cycles, stacked prisms, and particular types of caterpillars, by using various methods to construct odd prime labelings. We also demonstrate progress toward proving a conjecture that all prime graphs have an odd prime labeling.
\end{abstract}

\section{Introduction}

Consider a simple graph $G$ of order $n$.    A \textit{prime labeling} of $G$ is an assignment of the integers $1$ to $n$ as labels of the vertices such that the labels of each adjacent pair of vertices are relatively prime. A graph with such a labeling is called \textit{prime}. This labeling was conceived by Entringer and first introduced by Tout, Dabboucy, and Howalla~\cite{Tout}. By changing the labeling set or altering the relatively prime condition, many variations of prime labelings have since been developed, such as minimum coprime labelings~\cite{A_F, A_F2, BDHMMM, Lee}, neighborhood-prime labelings~\cite{C_F, Pa_Sh}, and $k$-prime labelings~\cite{V_P}. An overview of results on prime labelings and these variations can be found in the survey by Gallian~\cite{Gallian}. 

One of the more recently studied variations is the odd prime labeling~\cite{P_S, Y_A}, which consists of labeling the vertices with the set of odd integers $\{1,3,\ldots, 2n-1\}$ such that adjacent pairs of vertices have relatively prime labels. We use the term \textit{odd prime} to describe graphs that can be labeled in this manner.  A function $\ell$ will often be used as notation to describe the labeling as an injective map from the vertices $V$ to the set $\{1,3,\ldots, 2n-1\}$.  To be odd prime, we need $\gcd(\ell(u),\ell(v))=1$ for all adjacent vertices $u$ and $v$.  The first papers to introduce this labeling demonstrated a variety of classes of graphs are odd prime, including all paths, cycles, ladders, wheels, and the class of generalized Petersen graphs $GP(n,2)$~\cite{P_S}, as well as helms and disjoint unions of two cycles~\cite{Y_A}. Furthermore, characterizations for when complete graphs and complete bipartite graphs are odd prime are given in~\cite{Y_A}.

In this paper we continue the exploration of which classes of graphs admit an odd prime labeling. We first introduce some lemmas and observations in Section~\ref{prelim} that will aid in developing our labelings or determining that an odd prime labeling cannot exist. In Section~\ref{cycles} we investigate a wide array of graphs created through combining cycles such as the disjoint union of cycles, snake graphs, book graphs, and stacked prisms.  Section~\ref{trees} shifts our focus to the large class of connected graphs without cycles known as trees, in which we introduce labelings for spiders, perfect binary trees, special cases of caterpillars, and firecrackers.  In Section~\ref{powers} we fully characterize when odd prime labelings exist for powers of paths and cycles.  Finally, Section~\ref{pvop} consists of progress towards proving the conjecture that all prime graphs are also odd prime.

\section{Preliminary Material}\label{prelim}

The following lemmas summarize a few results by Youssef and Almoreed~\cite{Y_A} involving properties of odd prime graphs.  

\begin{lemma}\label{subgraph}
If $G$ is an odd prime graph, then every spanning subgraph of $G$ is also odd prime.
\end{lemma}

We use $\beta(G)$ to denote the independence number of a graph $G$, defined as the size of the largest set of vertices that is independent, meaning no edge exists between any pair of vertices in said set. 

\begin{lemma}\label{ind req}
The independence number of any odd prime graph $G$ of order $n$ satisfies
$$\beta(G)\geq \left\lfloor\frac{n+1}{3}\right\rfloor.$$
\end{lemma}

The lower bound of the previous inequality corresponds to the number of multiples of $3$ in the sequence of odd labels $1,3,\ldots, 2n-1$. Therefore, it provides a useful condition to demonstrate a graph is not odd prime if one can show $\beta(G)<\lfloor\frac{n+1}{3}\rfloor$ since in this case, there would not be enough independent vertices to place those multiples of $3$.  This is analogous to the requirement for prime graphs in which $\beta(G)\geq \lfloor\frac{n}{2}\rfloor$ where this relates to the placement of even labels from the sequence $1,2,\ldots, n$.

The following properties of the $\gcd$ of two integers will be very useful in proving our labelings are odd prime. Labels that differ by a positive power of two will be particularly helpful in many of our upcoming results and hence will be used without citation at times.
\begin{observation}\label{obs}
For any non-zero integers $a$, $b$, and $t$ and any positive integer $k$, the following hold:
\begin{enumerate}[(i)]
    \item $\gcd(a,a+2^k)=1$ if $a$ is odd,
    \item $\gcd(a,b)=\gcd(a,a-b)$,
    \item $\gcd(a,b)=\gcd(a+tb,b)$.
\end{enumerate}
\end{observation}

\section{Combinations of Cycles}\label{cycles}
We being our investigation of graphs with odd prime labelings by considering various ways of combining cycles.  First, we examine the disjoint union of cycles of any length, in which the disjoint union of two graphs is formed by taking the union of the vertex sets and the edge sets. It was shown in~\cite{Y_A} that the disjoint union $C_m\cup C_n$ is odd prime for any $m,n\geq 3$.  It is noteworthy that in the case of prime labelings, the disjoint union of cycles is not prime if both $m$ and $n$ are odd, which is a condition that does not carry over into odd prime labelings.  Furthermore, when generalized to the union of more than two cycles, Deretsky et al.~\cite{D_L_M} showed many cases of these disjoint unions to be prime if at most one odd cycle is included in the disjoint union.  However, it is ultimately still an open conjecture whether the disjoint union of any number of even cycles is always prime.  The following result proves the odd prime analog of this conjecture while also removing the restriction on the number of cycles being odd.  An example of a disjoint union of four cycles of various lengths with an odd prime labeling is shown in Figure~\ref{union of cycles}.

\begin{theorem}\label{unioncycles}
All disjoint unions of cycles, $\bigcup_{i=1}^n C_{k_i}$, are odd prime for any lengths $k_i$ and any number of cycles $n$.
\end{theorem}

\begin{proof}
We will label the cycles sequentially from $C_{k_1}$ to $C_{k_n}$. Given a cycle $C_k$ consider the vertices in clockwise order as $v_1, v_2, \ldots , v_k$. Let $m$ be the smallest unused odd label.  We create our labeling~$\ell$ starting with $\ell(v_1)$ in a pattern in which we alternate between the clockwise and counter-clockwise paths around the cycle.

If $k$ is even, label the sequence of vertices $$v_1, v_k, v_2, v_{k-1}, v_3, v_{k-2},\ldots, v_{\frac{k}{2}}, v_{\frac{k}{2}+1}$$ with the labels $m, m+2,\ldots , m+2k-2$, respectively. The vertices on edges $v_1v_k$ and $v_{\frac{k}{2}}v_{\frac{k}{2}+1}$ will be labeled with consecutive odd integers. All of the other adjacent labels will differ by~$4$.

If $k$ is odd, we similarly label the sequence of vertices  $$v_1, v_k, v_2, v_{k-1}, v_3, v_{k-2},\ldots , v_{\frac{k+1}{2}-1}, v_{\frac{k+1}{2}+1}, v_{\frac{k+1}{2}}$$ with $m, m+2,\ldots , m+2n-2$, respectively. The labels of the adjacent vertices again differ by $2$ or $4$.  Then by Observation~\ref{obs} $(i)$, adjacent labels in both cases are relatively prime. Thus all disjoint unions of cycles, $\bigcup_{i=1}^n C_{k_i}$, are odd prime for any lengths $k_i$ and any number of cycles $n$.
\end{proof}

\begin{figure}[t]
\begin{center}
\includegraphics[width=14 cm]{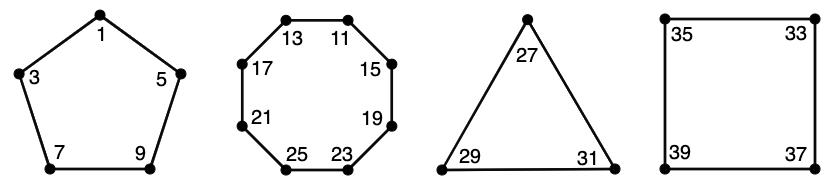}
\caption{An odd prime labeling of a disjoint union of cycles $C_5\cup C_8\cup C_3 \cup C_4$} \label{union of cycles} 
\end{center}
\end{figure}

We will now consider snake graphs, which are denoted as $S_{k,n}$ where $k$ represents the length of the cycles attached along a path of $n$ vertices.  Figure~\ref{snake} includes an odd prime labeling of a snake graph with hexagon cycles. The vertices along the path will be represented as $v_i$ with $i=1$ to $n$. The cycles are formed by adjoining paths $v_i, w_{i,1},w_{i,2},\ldots, w_{i,k-2},v_{i+1}$ for each $i=1$ to $n-1$.  Snake graphs were indirectly shown to be prime in~\cite{S_Y} where a related graph called the star$(m,n)$-gon was proven to be prime.  We now show the analogous is true for them being odd prime.

\begin{theorem}
Snake graphs $S_{k,n}$ are odd prime for any $k\geq 3$ and $n\geq 2$.
\end{theorem}

\begin{proof} Note that there are $(k-1)(n-1)+1$ vertices, so the largest label that we will use is $2(k-1)(n-1)+1$. We create our labeling function $\ell$ by assigning labels sequentially from 1 to $2(k-1)(n-1)+1$ to the following sequence of vertices: $$v_1,w_{1,1},w_{1,2},\ldots, w_{1,k-2},v_2,\ldots, v_{n-1},w_{n-1,1},w_{n-1,2},\ldots, w_{n-1,k-2},v_n.$$ The vertex pairs $w_{i,j}w_{i,j+1}$, as well as $v_iw_{i,1}$ and $w_{i,k-2}v_{i+1}$, will have relatively prime labels because they are labeled with consecutive odd integers. Therefore, the only pairs of labels that we have to confirm are relatively prime are on the vertices $v_i$ and $v_{i+1}$, where we note $\ell(v_i)=2(k-1)(i-1)+1$.
By Observation~\ref{obs} $(ii)$ for the second equality and $(iii)$ with $t=-i$ for the third equality, we have
\begin{align*}
\gcd(\ell(v_{i+1}),\ell(v_{i}))&=\gcd(2(k-1)i+1,2(k-1)(i-1)+1)\\
&=\gcd(2(k-1)i+1,2(k-1))\\
&=\gcd(1,2(k-1))\\
&=1.
\end{align*}
Thus all snake graphs $S_{k,n}$ are odd prime.
\end{proof}

\begin{figure}[t]
\begin{center}
\includegraphics[width=15 cm]{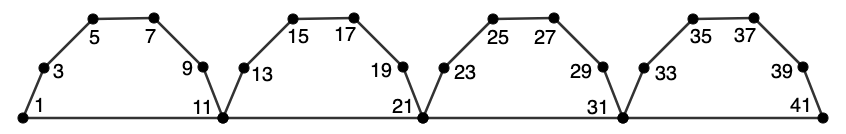}
\caption{An odd prime labeling of the snake graph $S_{6,5}$}\label{snake} 
\end{center}
\end{figure}

Our next class of graphs is similar to snake graphs in that $n$ cycles of length $k$ are linked together by shared vertices, but informally, they are glued together in the middle of the cycle instead of forming a path along their bases.  Cycle chains, denoted as $\mathcal{C}_k^n$, will be defined depending on the parity of the cycle length $k$. For even $k$, we begin with vertices $v_1,v_2,\ldots, v_{n+1}$ and form the following two paths between each pair: $v_i,w_{i,1}, w_{i,2},\ldots, w_{i,\frac{k}{2}-1},v_{i+1}$ and $v_i,x_{i,1}, x_{i,2},\ldots, x_{i,\frac{k}{2}-1},v_{i+1}$.  When $k$ is odd, we create the cycles in $\mathcal{C}_k^n$ in similar fashion, but the $w$-path will be one edge shorter than the $x$-path.  Specifically, we join the vertices $v_i$ and $v_{i+1}$ with the paths $v_i,w_{i,1}, w_{i,2},\ldots, w_{i,\frac{k-1}{2}-1},v_{i+1}$ and $v_i,x_{i,1}, x_{i,2},\ldots, x_{i,\frac{k-1}{2}},v_{i+1}$.  See Figure~\ref{cycle chain} for an example of a cycle chain with an odd prime labeling.  Cycle chains, particularly in the case of even length cycles, were examined for prime labelings in~\cite{D_E} where they were shown to be prime in certain cases, but we conclude in the next result that all cycle chains have an odd prime labeling. 

\begin{figure}[t]
\begin{center}
\includegraphics[width=12 cm]{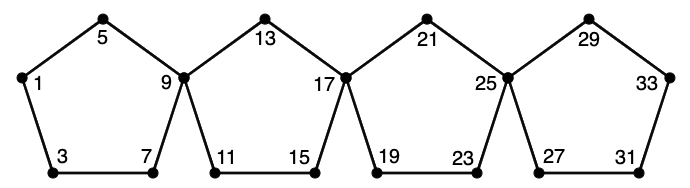}
\caption{An odd prime labeling of the cycle chain graph $C_5^4$}\label{cycle chain} 
\end{center}
\end{figure}

 
\begin{theorem}
Cycle chains, $\mathcal{C}_k^n$, are odd prime for all $k\geq 3$ and $n\geq 1$. 
\end{theorem}

\begin{proof}
We start by labeling $v_1$ with $1$ and then follow a similar alternating pattern as in the labeling of the disjoint union of cycles. If the cycle is even, we then label the vertices in this pattern for all $1\leq i\leq n$ :
$$ w_{i,1}, x_{i,1}, w_{i,2}, x_{i,2},\ldots , w_{i, \frac{k}{2}-1}, x_{i, \frac{k}{2}-1}, v_{i+1}$$
with the odd integers $1,3,\ldots, 2nk-2n+1$.  If the cycle is odd, label the vertices as follows for all $1\leq i\leq n$:
$$x_{i,1}, w_{i,1}, x_{i,2}, w_{i,2},\ldots , w_{i,\frac{k-3}{2}}, x_{i,\frac{k-1}{2}}, v_{i+1}$$
with the odd labels $1,3,\ldots,2nk-2n+1$.  Since every difference between the labels of adjacent vertices will be 2 or 4, all cycle chains are odd prime.
\end{proof}

We next examine the book graph $B_{n,k}$, defined as $n$ cycles each with $k$ vertices that all share a common edge. Each cycle is referred to as a page of the book graph.  We call the vertices of the shared edge $u$ and $v$.  Let $w_{i,j}$ be the vertices in the $i$th page with $j$ being the clockwise position within each cycle, such that $u$ is adjacent to each $w_{i,1}$ and $v$ is adjacent to $w_{i,k-2}$. Figure~\ref{book} includes a book graph featuring hexagonal pages with an odd prime labeling.  This particular class of graphs was shown to be prime in~\cite{S_Y} for the case of square pages or $k=4$.  We now prove the broader result that book graphs are odd prime for any size or number of pages.

\begin{figure}[t]
\begin{center}
\includegraphics[width=7 cm]{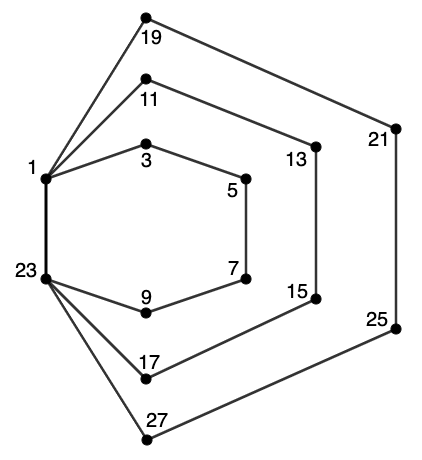}
\caption{An odd prime labeling of the book graph $B_{3,6}$}\label{book} 
\end{center}
\end{figure}

\begin{theorem}
Book graphs $B_{n,k}$ are odd prime for any $n\geq 1$ and $k\geq 3$.
\end{theorem}

\begin{proof}
We first assign $\ell(u)=1$, and then label the sequence of vertices in the first page $s_{1,1}$ to $s_{1,k-2}$ with $3,5,\ldots, 2k-3$, respectively. Continue the pattern of using the next $k-2$ odd integers to label the subsequent pages from $i=2$ to $n$. This leaves the vertex $v$ to be labeled as the next available odd integer so that $\ell(v)=2(k-2)n+3$.

After this initial labeling, the current label of $v$ may share a common factor with the label of one or more of the vertices $\ell(w_{i,k-2})$ for some $i$. Therefore, we will appeal to Bertrand's postulate, which states for any integer $t>1$, there always exists at least one prime number $p$ with $t<p<2t$. Thus, considering $t$ to be the number of vertices in $B_{n,k}$, which is $(k-2)n+2$, we can guarantee that there will be a prime $p$ such that $(k-2)n+2<p<2(k-2)n+4$. This means that we have already used this prime number within our labeling.  If $p=2(k-2)n+3$, then we can leave this as the label of~$v$.  Otherwise, $\ell(w_{a,b})=p$ for some particular $a$ and $b$, so we reassign $\ell(v)=p$ and shift the labels on page $a$ so that $\ell(w_{a,k-2})=2(k-2)n+3$ and each $\ell(w_{a,j})$ for $j=b$ to $k-3$ is assigned the original label of $w_{a,j+1}$.  

We clearly have that $\gcd(\ell(u),\ell(w_{i,1}))=1$ since $\ell(u)=1$.  Since $p$ is a prime number that is large enough so that $2p$ is greater than any other label on the graph, $\gcd(\ell(v),\ell(w_{i,k-2}))=1$ as well.  The labels on vertices $w_{a,b-1}$ and $w_{a,b}$ differ by $4$ after the labeling shift has been made, and all adjacent pairs of the form $w_{i,j}$, $w_{i,j+1}$ have labels that are consecutive odd integers. Thus, our labeling of the book graph $B_{n,k}$ is odd prime. 
\end{proof}

Next, we will investigate prism graphs. We define prism graphs as two cycles each with $n$ vertices, the outside cycle with vertices $v_1$ to $v_n$ and the inside cycle with vertices  $u_1$ to $u_n$. The edges are defined as $u_iu_{i+1}$, $v_iv_{i+1}$ for $i=1$ to $n-1$, $u_nu_1$, $v_nv_1$, and $u_iv_i$ for $i=1$ to $n$. Prism graphs are also a type of generalized Peterson graph, so we use the notation $GP(n,1)$.  See Figure~\ref{Prism} for an example of an odd prime labeling of a prism graph.  Although we will show all prism graphs have odd prime labelings, this is not the case for prime labelings.  When $n$ is odd, the prism graph is not prime and has only been shown to be prime in certain cases for even $n$, as seen in~\cite{HLYZ}.

\begin{figure}[t]
\begin{center}
\includegraphics[width=7 cm]{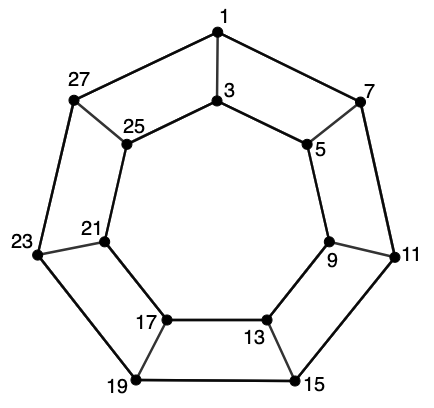}
\caption{An odd prime labeling of the prism graph GP(7,1)}\label{Prism} 
\end{center}
\end{figure}

\begin{theorem}\label{prismgraphs}
The prism graph $GP(n,1)$ is odd prime for all $n\geq 3$.
\end{theorem}

\begin{proof}
We begin by assigning the labels as follows: $\ell(u_i)=4i-3$ and $\ell(v_i)=4i-1$ for $i=1$ to $n$. For all edges of the form $u_iv_i$, we have $\gcd(\ell(u_i),\ell(v_i))=1$ because the labels are consecutive odd integers.  Likewise, the labels on the endpoints of the edges $v_iv_{i+1}$ and $u_iu_{i+1}$, for $i=1$ to $n-1$, will be relatively prime because they differ by $4$. 
 
 The last edges to consider are $u_nu_1$ and $v_nv_1$. Since  $\ell(u_1)=1$, it will always be relatively prime with $\ell(u_n)$. 
 If $n\not\equiv 1\pmod{3}$, then $\ell(v_n)=4n-1$ is not a multiple of $3$ and thus $\gcd(\ell(v_n),\ell(v_1))=\gcd(4n-1,3)=1$.  This proves that $\ell$ is an odd prime labeling in this case. 
 If $n\equiv 1\pmod{3}$ then 3 divides $\ell(v_n)=4n-1$, so a swap must be made in our labeling. We reassign $\ell(u_1)=3$ and $\ell(v_1)=1$. 
 Now we have $\gcd(\ell(v_n),\ell(v_1))=1$ since $\ell(v_1)=1$, and 
 $\gcd(\ell(u_n),\ell(u_1))=\gcd(4n-3,3)=1$ since $3$ does not divide $4n-3$ when $n\equiv 1\pmod{3}$.  Note that we still have $\gcd(\ell(u_1),\ell(u_2))=\gcd(3,5)=1$ and $\gcd(\ell(v_1),\ell(v_2))=\gcd(1,7)=1$.  Therefore, we have created an odd prime labeling in this case of $n\equiv 1\pmod{3}$, which proves that all prisms are odd prime.
 
\end{proof}

Prisms are a special case with $k=1$ of the generalized Petersen graph, $GP(n,k)$, where $n\geq 3$ and $1\leq k\leq \lfloor (n-1)/2\rfloor$.  This graph has $2n$ vertices $v_1,v_2,\ldots,v_n,u_1,u_2,\ldots, u_n$ and $3n$ edges of the forms $v_iv_{i+1}$, $u_iu_{i+k}$, and $v_iu_i$ with indices calculated modulo $n$.  It has been shown in~\cite{P_S} that $GP(n,2)$ is odd prime for any $n$ and conjectured in the same paper $GP(n,k)$ is odd prime for any valid $n$ and~$k$.  While $GP(n,k)$ is not prime for any odd $k$, Theorem~\ref{Prism} add evidence to their odd prime conjecture, although it remains open for $k>2$.

We next turn our focus onto stacked prism graphs, denoted by $Y_{k,n}$ with $k\geq 3$ and $n\geq 1$, which is a generalization of prisms graphs.  They consist of $n$ polygons each with $k$ vertices that are connected by edges between consecutive polygons.  More formally, $Y_{k,n}$ is the Cartesian product of a cycle of length $k$ with a path of length $n$, or $C_k\square P_n$.  We will show triangular, pentagonal, and hexagonal stacked prisms have odd prime labelings. Additionally, we will extend these results for stacked prisms with $2^m$ sides on the polygons.  With respect to prime labelings, $Y_{k,n}$ is not prime for any odd~$k$. No work on the even $k$ case has been previously published, although the independence number of $Y_{k,n}$ for even $k$ would lead one to believe they are prime.
  
First, we will consider triangular stacked prisms, $Y_{3,n}$. We use $v_{i,j}$ with $1\leq j\leq3$ to refer to the vertices on the $i$th triangle. The edges are defined as $v_{i,1}v_{i,2}$, $v_{i,2}v_{i,3}$, and $v_{i,3}v_{i,1}$ for $i=1,2,\ldots, n$ and $v_{i,j}v_{i+1,j}$ with $1\leq i\leq n-1$ and $1\leq j\leq 3$.  Figure~\ref{Triangluar} shows an example of an odd prime labeling on a triangular stacked prism with $n=4$.

\begin{figure}[t]
\begin{center}
\includegraphics[width=10 cm]{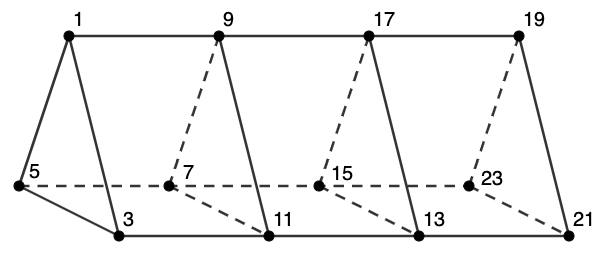}
\caption{An odd prime labeling of a triangular stacked prism $Y_{3,4}$ }\label{Triangluar} 
\end{center}
\end{figure}

\begin{theorem}\label{triangular}
The triangular stacked prism, $Y_{3,n}$, is odd prime for all $n \geq 1$.
\end{theorem}

\begin{proof}
We will refer to the table below to label this graph. For any $v_{i,j}$ with $i\leq 3$, let $\ell(v_{i,j})$ be the corresponding value in the table. For any $v_{i,j}$ with $i>3$, this means that $i=3a+b$ for some positive integer $a$ and some integer $b$ with $0\leq b\leq 2$. Then we assign $\ell(v_{i,j})=\ell(v_{b,j})+18a$.

\begin{center}
\begin{tabular}{ |p{2cm}|p{2cm}|p{2cm}|p{2cm}| }
  \hline
  & $v_{i,1}$ & $v_{i,2}$ & $v_{i,3}$  \\
  \hline
  $i=1$ & 1 & 3 & 5 \\
  \hline 
  $i=2$ & 9 & 11 & 7 \\
  \hline
  $i=3$ & 17 & 13 & 15 \\
  \hline 
\end{tabular}
\end{center}

By inspecting the table above, for $i=1$ to $3$, we can see that the difference between the labels of the adjacent vertices $v_{i,j}$ and $v_{i+1,j}$ is always $2$ or $8$. Similarly, the difference between the labels of adjacent vertices $v_{i,j}$ and $v_{i,k}$ within the same triangle will be $2$ or $4$. Considering one more level in the stacked prism, when $i=4$, the labels of this triangle are shifted by 18 from the first triangle to be $\ell(v_{4,1})=19$, $\ell(v_{4,2})=21$, and $\ell(v_{4,3})=23$.  This implies the difference between labels of the vertices~$v_{3,j}$ and $v_{4,j}$ are either $2$ or $8$.

Note that for $i>4$ the adjacent labels will remain relatively prime because adding the same multiples of $18$ will preserve the differences in the labels that were all powers of $2$. Thus all triangular stacked prism graphs have an odd prime labeling.  

\end{proof}

Pentagonal stacked prisms, $Y_{5,n}$, are very similar to triangular stacked prisms as the name suggests, but with pentagons in each layer of the prism.  We refer to the vertices on the $i$th pentagon as $v_{i,j}$ with $1\leq j\leq5$. The edges are defined for each $i=1,2,\ldots, n$ as $v_{i,j}v_{i,j+1}$ with $1\leq j\leq 4$ and $v_{i,5},v_{i,1}$, and for $i=1,2,\ldots, n-1$ as $v_{i,j},v_{i+1,j}$ with $1\leq j\leq 5$ .  An example of a pentagonal stacked prism with an odd prime labeling is displayed in Figure~\ref{Pentagonal}.

\begin{figure}[t]
\begin{center}
\includegraphics[width=17 cm]{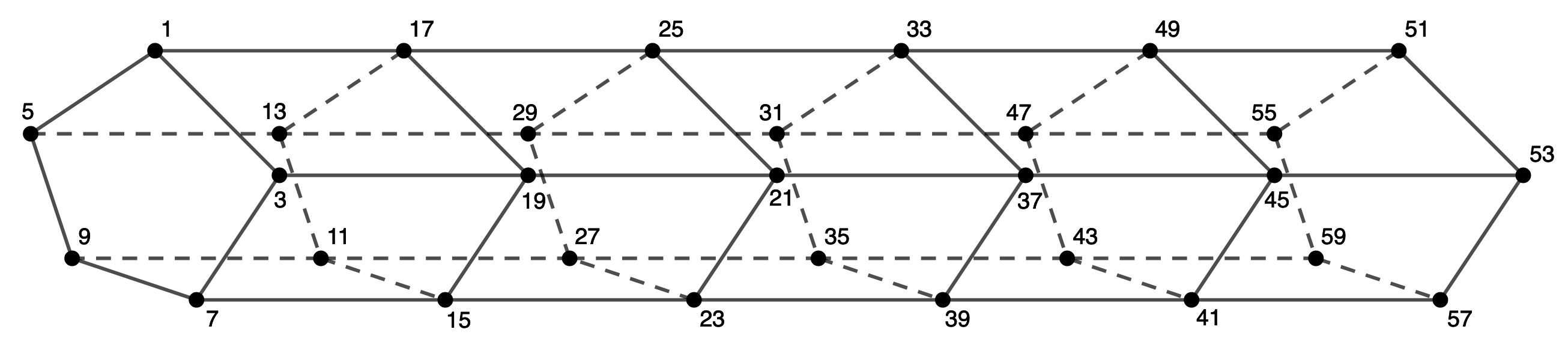}
\caption{An odd prime labeling of a pentagonal stacked prism graph $Y_{5,6}$}\label{Pentagonal} 
\end{center}
\end{figure}

\begin{theorem}\label{pentagonal}
The pentagonal stacked prism, $Y_{5,n}$ is odd prime for all $n \geq 1$.
\end{theorem}

\begin{proof}
As in the last proof, we will provide a table of labels for the first several layers of the stacked prism. For any $v_{i,j}$ with $i\leq5$, we assign $\ell(v_{i,j})$ to be the value in the table below. For any $v_{i,j}$ with $i>5$ we have $i=5a+b$ for some positive integer $a$ and integer $b$ with $0\leq b\leq 4$. Then assign $\ell(v_{i,j})=\ell(v_{b,j})+50a$. 

\begin{center}
\begin{tabular}{ |p{2cm}|p{2cm}|p{2cm}|p{2cm}|p{2cm}|p{2cm}| }
  \hline
  & $v_{i,1}$ & $v_{i,2}$ & $v_{i,3}$ & $v_{i,4}$ & $v_{i,5}$ \\
  \hline
  $i=1$ & 1 & 3 & 7 & 9 & 5 \\
  \hline 
  $i=2$ & 17 & 19 & 15 & 11 & 13 \\
  \hline
  $i=3$ & 25 & 21 & 23 & 27 & 29 \\
  \hline 
  $i=4$ & 33 & 37 & 39 & 35 & 31 \\
  \hline
  $i=5$ & 49 & 45 & 41 & 43 & 47 \\
  \hline 
\end{tabular}
\end{center}

By inspection of the table, it is clear that the difference between the labels of adjacent vertices~$v_{i,j}$ and $v_{i+1,j}$ is either $2$, $8$ or $16$. We can also see that the difference between labels of adjacent vertices $v_{i,j}$ and $v_{i,j+1}$ for $1\leq j\leq 4$ and $v_{i,5}$ and $v_{i,1}$ will always be $2$ or $4$.  Adjacent vertices from the fifth pentagon and the sixth one after shifting the labels by 50 to get $\ell(v_{6,j})=\ell(v_{1,j})$ also have labels that differ by $2$, $8$, or $16$.

For $i>5$, shifting by the same multiple of $50$ will preserve the differences in adjacent labels to keep them as powers of $2$. Thus all pentagonal stacked prisms have an odd prime labeling.

\end{proof}

We continue with another type of stacked prism graph, $Y_{6,n}$, the hexagonal stacked prism graph. The vertices on the $i$th hexagon are referred to as $v_{i,j}$ with $1\leq j\leq6$. The edges are defined for $i=1,2,\ldots, n$ as $v_{i,j},v_{i,j+1}$ with $1\leq j\leq 5$ and $v_{i,6},v_{i,1}$, and for $i=1,2,\ldots, n-1$ as $v_{i,j},v_{i+1,j}$ with $1\leq j\leq 6$. Figure~\ref{hexagonal} includes an example of a hexagonal stacked prism with an odd prime labeling.

\begin{figure}[t]
\begin{center}
\includegraphics[width=16 cm]{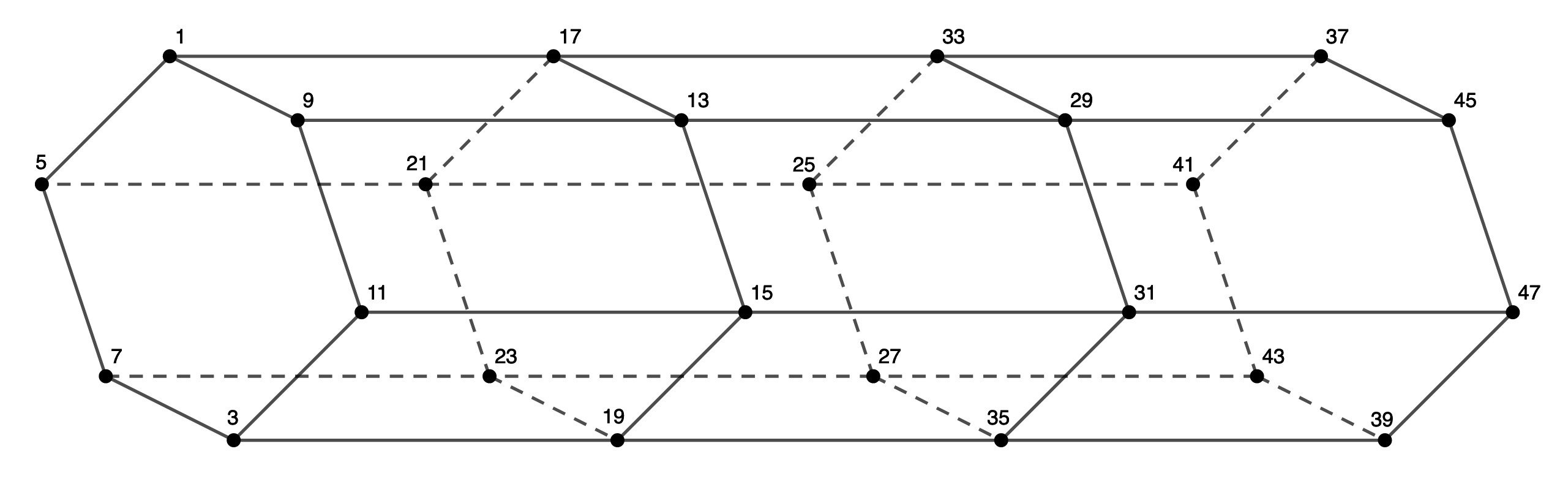}
\caption{An odd prime labeling of a hexagonal stacked prism $Y_{6,4}$ }\label{hexagonal} 
\end{center}
\end{figure}

\begin{theorem}
The hexagonal stacked prism, $Y_{6,n}$, is odd prime for all $n\geq 1$.
\end{theorem}

\begin{proof}
As in previous stacked prism results, for $v_{i,j}$ with $i\leq 3$, we assign $\ell(v_{i,j})$ as the value in the table below. For $i>3$, we have $i=3a+b$ for some positive integer $a$ and an integer $b$ with $0\leq b\leq 2$. In this case, we assign $\ell(v_{i,j})=\ell(v_{b,j})+36a$. 

\begin{center}
\begin{tabular}{ |p{2cm}|p{2cm}|p{2cm}|p{2cm}|p{2cm}|p{2cm}|p{2cm}| }
  \hline
  & $v_{i,1}$ & $v_{i,2}$ & $v_{i,3}$ & $v_{i,4}$ & $v_{i,5}$ &$v_{i,6}$ \\
  \hline
  $i=1$ & 1 & 9 & 11 & 3 & 7 & 5 \\
  \hline 
  $i=2$ & 17 & 13 & 15 & 19 & 23 & 21 \\
  \hline
  $i=3$ & 33 & 29 & 31 & 35 & 27 & 25 \\
  \hline 
\end{tabular}
\end{center}

One can see from the table that the difference between the labels of adjacent vertices $v_{i,j}$ and $v_{i+1,j}$ is $4$ or $16$. Meanwhile, the difference between labels of adjacent vertices $v_{i,j}$ and $v_{i,j+1}$ for $1\leq j\leq 5$ and $v_{i,6}$ and $v_{i,1}$ will always be $2, 4$ or $8$.  When considering the edges between the third and fourth hexagons, the labels after shifting by $36$ on the fourth hexagon are $37$, $45$, $47$, $39$, $43$, and $41$, creating differences of $4$ and $16$.

For $i>3$ adjacent labels will continue to be relatively prime because adding the same multiples of $36$ will preserve the differences from the earlier cases of $i$. Thus all hexagonal stacked prisms have an odd prime labeling.

\end{proof}

We now expand our results on stacked prisms to polygons where the number of sides are any power of $2$. As with our previous cases, we refer to the vertices on the $i$th polygon as $v_{i,j}$ where $1\leq j\leq2^m$. The edges are the following: $v_{i,j}, v_{i,j+1}$ with $1\leq j\leq 2^m -1$ and $v_{i,1}, v_{i,j}$ for $1\leq i \leq n$, and $v_{i,j}, v_{i+1,j}$ for $1\leq j\leq 2^m$ and $1\leq i\leq n-1$.

\begin{theorem}
The $2^m$-sided polygonal stacked prism, $Y_{2^m,n}$, is odd prime for all $m\geq 2$ and $n\geq 1$.
\end{theorem}
\begin{proof}
We label the vertices on the $i$th polygon with the values from $(i-1)2^{m+1}+1$ to $i\cdot 2^{m +1}-1$ in this order:
$$v_{i,1}, v_{i,2^m}, v_{i,2}, v_{i,2^m -1}, v_{i,3},\ldots, v_{i,2^{m-1}+2}, v_{i,2^{m-1}},  v_{i,2^{m-1}+1}.$$
Using this alternating pattern will ensure that all adjacent labels on each polygon will be separated by a power of $2$, particularly $2$ or $4$.  The difference between the labels $\ell(v_{i,1})$ and $\ell(v_{i+1,1})$ is
$$i\cdot 2^{m+1}+1-((i-1)2^{m+1}+1)=2^{m+1}.$$ Likewise, all other edges of the form $v_{i,j}v_{i+1,j}$ have labels differing $2^{m+1}$ as well. Hence all adjacent labels are relatively prime by Observation~\ref{obs} (i), making all $2^m$-sided polygonal stacked prisms odd prime.
\end{proof}

As one considers generalizing our previous results on stacked prisms to attempt to create an odd prime labeling of $Y_{k,n}$ for any $k\geq 3$, it is worth noting that the independence number of this graph is $\lfloor\frac{k}{2}\rfloor\cdot n$.  This is well above the required independence number from Lemma~\ref{ind req} of $\lfloor \frac{kn+1}{3}\rfloor$ to make an odd prime labeling possible.  

We believe our approach would work for any $k$ to find a permutation of the labels $1,3,\ldots, 2k-1$ for the first $k$-polygon, followed by a mapping to match these to the labels $2k+1,2k+3,\ldots, 4k-1$ with each difference being a power of 2.  The mapping could then be repeated for each subsequent $k$-polygon.  However, constructing this mapping for a general $k$ has eluded us, leaving the following as a conjecture.

\begin{conjecture}
The $k$-polygonal stacked prism, $Y_{k,n}$, is odd prime for all $k\geq 3$ and $n\geq 1$.
\end{conjecture}

One well-studied class of graphs within the field of prime labeling is the set of grid graphs, $P_m\times P_n$.  It is conjectured that $P_m\times P_n$ is prime for all $m,n\geq 2$, and many partial results have been proven for certain cases of $m$ and $n$~\cite{V_S_N,S_P_S}, and in particular ladder graphs, $P_m\times P_2$, have been proven to be prime for any length $m$~\cite{Dean}. Ladder graphs were shown in~\cite{Y_A} to be odd prime, but grid graphs with higher values of $m$ have yet to be investigated for odd prime labelings. 

Our results on stacked prisms easily translate into odd prime labelings of grid graphs since $P_m\times P_n$ is the graph $Y_{m,n}$ with the edges between vertices $v_{i,n}$ and $v_{i,1}$ removed. Therefore, Lemma~\ref{subgraph} directly proves the following. 

\begin{corollary}
Grid graphs $P_m\times P_n$ are odd prime for $m=3, 5, 6,$ and $2^k$ with $k\geq 2$ and for any $n$.
\end{corollary}

\section{Trees}\label{trees}
Various classes of trees have been shown to have prime labelings, providing evidence for the conjecture by Entringer that all trees are prime.  With regards to odd prime labelings, only paths~\cite{P_S}, unions of paths, and star graphs (or $K_{1,n}$)~\cite{Y_A} have been shown to be odd prime.  We will construct odd prime labelings for several other classes of trees to support the following conjecture.

\begin{conjecture}
All trees are odd prime.
\end{conjecture}

We begin with two types of trees that were shown in~\cite{F_H} to have prime labelings. Spiders are a class of trees in which only one vertex has degree 3 or more.  It can also be viewed as a collection of paths $P_{n_1}, P_{n_2},\ldots,P_{n_k}$ with one end of each path adjoined to a central vertex.  Figure~\ref{spider} shows an example of spider with an odd prime labeling.

\begin{figure}[t]
\begin{center}
\includegraphics[width=7 cm]{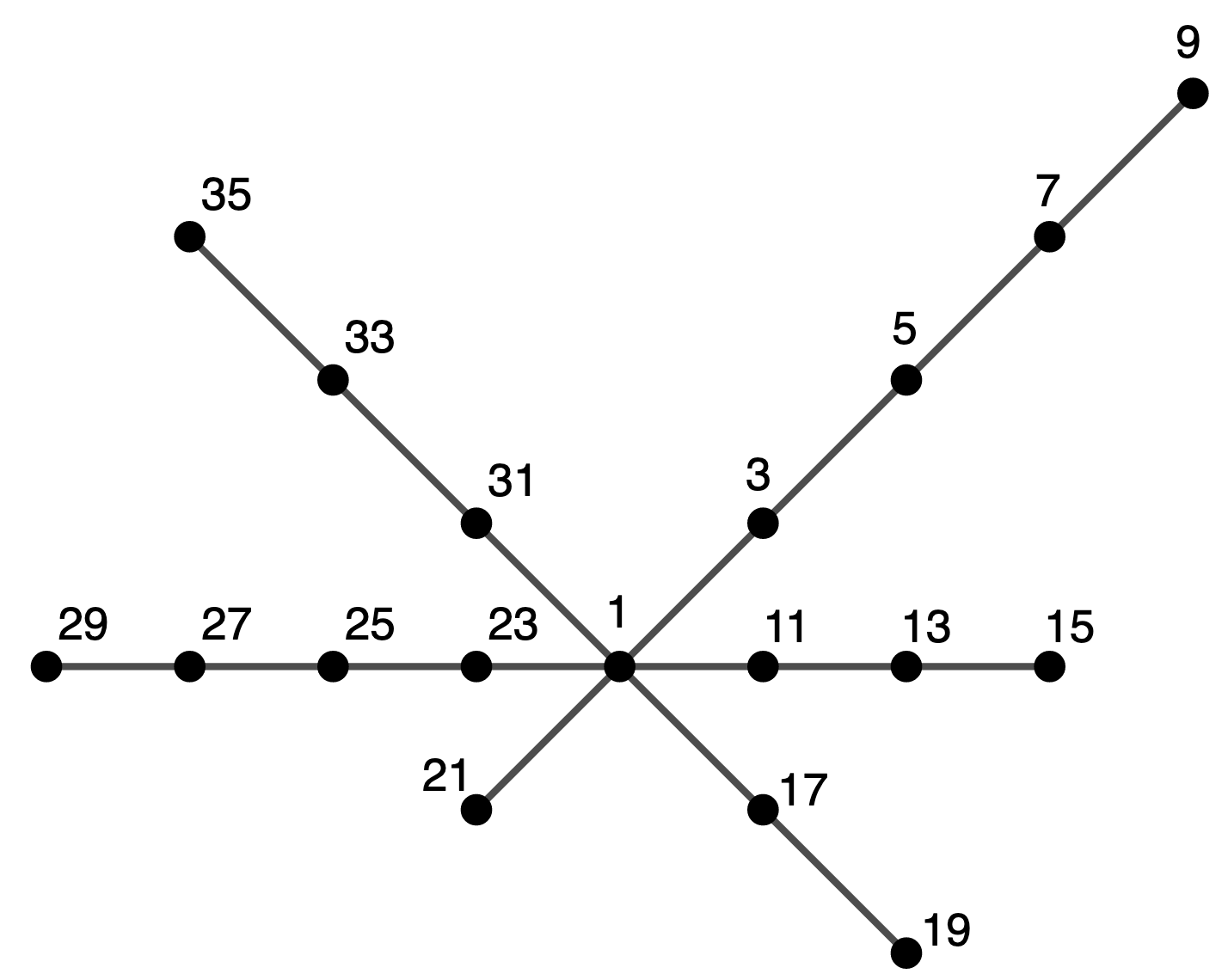}
\caption{An odd prime labeling of a spider graph}\label{spider} 
\end{center}
\end{figure}

\begin{theorem}
All spider graphs are odd prime.
\end{theorem}
\begin{proof}
Begin by assigning the label 1 to the central vertex, which we will call $v$.  Then label the vertices of $P_{n_1}$ with $3,5,\ldots, 2(n_1+1)-1$ in order from the vertex adjacent to the $v$ outward to the leaf.  Continue labeling the other paths $P_{n_2}, P_{n_3},\ldots, P_{n_k}$ similarly with the lowest available label assigned to the vertex adjacent to $v$.  This labeling is odd prime since every pair of adjacent vertices within a path $P_{n_{i}}$ contains consecutive odd labels, and any adjacent pair that includes $v$ involves the label~$1$.
\end{proof}

Next we consider perfect binary trees, which are rooted trees in which all interior vertices have two children and all leaves are on the same level.  Note that these trees have been proven to be prime in~\cite{F_H}, but were referred to as a complete binary tree.  See Figure~\ref{binary} for an example of a perfect binary tree with an odd prime labeling.

\begin{figure}[t]
\begin{center}
\includegraphics[width=15 cm]{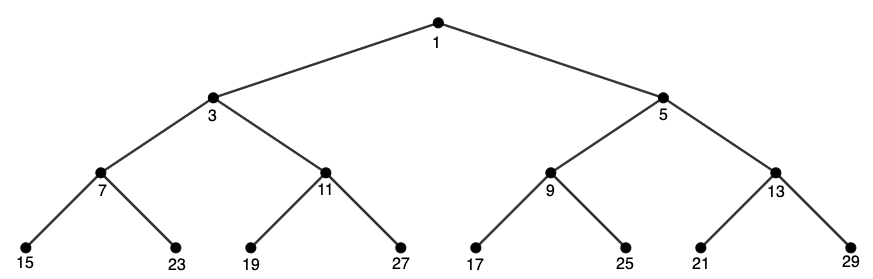}
\caption{An odd prime labeling of a perfect binary tree with four levels}\label{binary} 
\end{center}
\end{figure}

\begin{theorem}
All perfect binary trees are odd prime.
\end{theorem}
\begin{proof}

Assume $n$ is the number of levels of the binary tree, and consider the vertices of level $i$ to be $v_{i,j}$, where $i=1, 2,\ldots, n$ and $j=1, 2,\ldots, 2^{i-1}$. We start with the following labels: $\ell(v_{1,1})=1$, $\ell(v_{2,1})=3$, and $\ell(v_{2,2})=5$.  Clearly the adjacent labels for the first two levels are relatively prime. We will proceed by induction on $n$.  Assume that the first levels $1$ to $k$ are labeled such that all adjacent pairs are assigned relatively prime labels and with level $k$ containing the labels $\{2^k-1$, $2^k+1$, \ldots, $2^{k+1}-3 \}$.  Note that our chosen labels for levels $1$ and $2$ also meet this second condition, so they can be viewed as our base case. 

We now label level $k+1$.  For each $v_{k,j}$ with $\ell(v_{k,j})=x$, we label its children $x+2^k$ and $x+2^{k+1}$. Adding $2^k$ to the labels of row $k$ results in the labels $\{2^{k+1}-1, 2^{k+1}+1,\ldots, 2^{k+1}+2^k-3\}$.  Similarly, adding $2^{k+1}$ generates labels $\{2^{k+1}+2^k-1, 2^{k+1}+2^k+1,\ldots, 2^{k+2}-3\}$. These $2$ sets of labels are disjoint and when combined form the set $\{2^{k+1}-1, 2^{k+1}+1,\ldots,  2^{k+2}-3\}$. The labels on $v_{i,j}$ and each child are relatively prime because they differ by $2^k$ or $2^{k+1}$.  Thus, by induction, the labeling $\ell$ is an odd prime labeling for any number of rows $n$.

\end{proof}

We next explore caterpillars, which are defined as a tree with a central path of $n$ vertices called the spine, and all other vertices are distance $1$ from the $n-2$ interior vertices of the spine.  Two particular classes of caterpillars were examined in the initial prime labeling paper by Tout et al~\cite{Tout}, the first of which being caterpillars with the highest degree vertex being bounded above by $5$.  Figure~\ref{caterpillar} includes such a caterpillar with an odd prime labeling, and the following result describes a method of labeling these caterpillars with small degree vertices.

\begin{figure}[t]
\begin{center}
\includegraphics[width=15 cm]{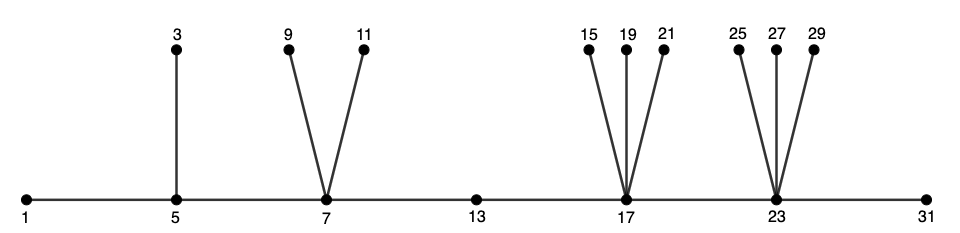}
\caption{An odd prime labeling of a caterpillar with vertices having at most degree 5}\label{caterpillar} 
\end{center}
\end{figure}

\begin{theorem}
All caterpillars with maximum degree at most $5$ are odd prime.
\end{theorem}

\begin{proof}
Assume the caterpillar has $n$ vertices on its spine, and let those vertices be $v_1$, $v_2$,\ldots,$v_n$. Begin by labeling $v_1$ as $\ell(v_1)=1$, so $\gcd(\ell(v_1),\ell(v_2))=1$ regardless of what the label of $v_2$ is. We will sequentially label each $v_i$ from $i=2$ to $n-1$ along with its adjacent leaves, where we consider $m$ to be the lowest unused odd integer thus far.
\begin{itemize}
\item[--] If $\deg(v_i)=2$, assign $\ell(v_i)=m$. 

\item[--] If $\deg(v_i)=3$, then at least one of $m$ or $m+2$ is not a multiple of $3$.  Assign the smallest such value as $\ell(v_i)$ and the other label to its leaf. 

\item[--]  If $\deg(v_i)=4$, then one of $m$, $m+2$, and $m+4$ is a multiple of $3$, and at most one is a multiple of $5$.  Therefore, at least one of $m$, $m+2$, and $m+4$ is neither a multiple of $3$ or $5$. Assign the smallest such value as $\ell(v_i)$ and the other two labels to its leaves. 

\item[--] If $\deg(v_i)=5$, at most two of $m$, $m+2$, $m+4$, and $m+6$ are a multiple of $3$, and at most one is a multiple of $5$.  Hence at least one of these is not a multiple of $3$ or $5$. Assign the smallest such integer as $\ell(v_i)$, and the remaining labels to its three leaves. 

\end{itemize}

Finally we assign $\ell(v_n)$ to be the next available odd integer. For each edge $v_iu$ where $u$ is a leaf, $\gcd(\ell(v_i),\ell(u))=1$ since their difference is $2$ or $4$ with one possible exception.  In the case of $\deg(v_i)=5$, we may have that the adjacent labels are $m$ and $m+6$. Since the label of $v_i$ is selected to not be a multiple of $3$, then we know $\gcd(m,m+6)=1$ as well. 

For edges of the form $v_iv_{i+1}$ with $i=2, 3, \ldots, n-2$, we consider each possible common prime factor.  The only case when $3$ divides $\ell(v_i)$ is if $\deg(v_i)=2$, since we avoided multiples of $3$ on the spine with higher degree vertices.  Then $3$ can only be a common factor of $\ell(v_i)$ and $\ell(v_{i+1})$ if $\deg(v_i)=\deg(v_{i+1})=2$, but in that case $\ell(v_i)$ and $\ell(v_{i+1})$ would be consecutive odd integers and hence won't have a common factor of $3$. 

A common factor of $5$ would require $\deg(v_i)$ and $\deg(v_{i+1})$ to be at most $3$ since multiples of $5$ were avoided for vertices with degree $4$ or $5$. This implies the greatest difference between $\ell(v_i)$ and $\ell(v_{i+1})$ could be $6$, but a difference of at least $10$ would be necessary to share a common factor of $5$.

In order to have a common factor of $7$, we would need $\ell(v_{i+1})-\ell(v_i)=14$. This can only occur if $\deg(v_i)=\deg(v_{i+1})=5$ with $\ell(v_i)=m$ and $\ell(v_{i+1})=m+14$, where the leaves of $v_{i+1}$ would be labeled as $m+8$, $m+10$, and $m+12$. However, at least one of those $3$ labels is not a multiple of $3$ or $5$, and being less than $m+14$, it would have been assigned as $\ell(v_{i+1})$.  This precludes $\ell(v_i)$ and $\ell(v_{i+1})$ from having a common factor of $7$. 

Since our labeling is created with the largest difference possible in the labels for $v_i$ and $v_{i+1}$ being $14$, their labels cannot share a prime factor larger than $7$.  Thus, $\gcd(\ell(v_i),\ell(v_{i+1}))=1$. for $i=2,3,\ldots, n-2$.

Lastly, we must justify why $v_{n-1}$ and $v_n$ are labelled by relatively prime integers. The construction of our labeling implies that the only way a multiple of $3$ would be a label on the spine is if the degree of $v_{n-1}$ was $2$, but then $\ell(v_n)$ would be the next consecutive odd integer and thus $\gcd(\ell(v_{n-1}), \ell(v_n))=1$. If $\deg(v_{n-1})>2$, a non-multiple of $3$ would have been chosen as its label. Then if the $\ell(v_n)$ is a multiple of $3$, we know $\ell(v_{n-1})$ would not be a multiple of $3$. For common prime factors of 5 or higher, we note that since $\deg(v_{n-1})\leq 5$, the largest difference between the labels of $v_{n-1}$ and $v_n$ is 8. Thus no common factors of 5 or higher can exist, implying $\gcd(\ell(v_{n-1}), \ell(v_n))=1$.

Therefore, all adjacent vertices have relatively prime labels, so there exists an odd prime labeling for all caterpillars with maximum degree at most 5.
\end{proof}

We next examine a second set of caterpillars in which the interior vertices of the tree have the same degree.  Referring to the leaves of these caterpillars as toes, we refer to a $t$-toed caterpillar as one where each interior vertex has degree $t+2$. Two examples of $t$-toed caterpillars can be seen in Figure~\ref{caterpillar2} complete with an odd prime labelings.  

\begin{figure}[t]
\begin{center}
\includegraphics[width=14 cm]{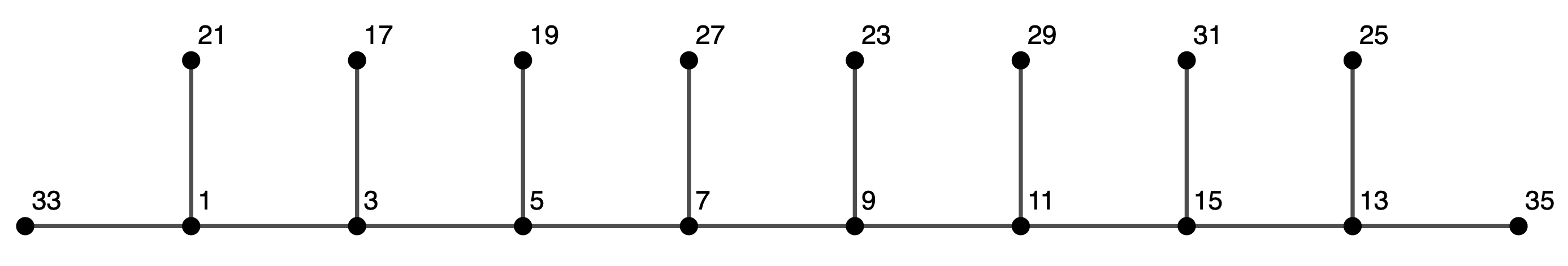}
\includegraphics[width=11 cm]{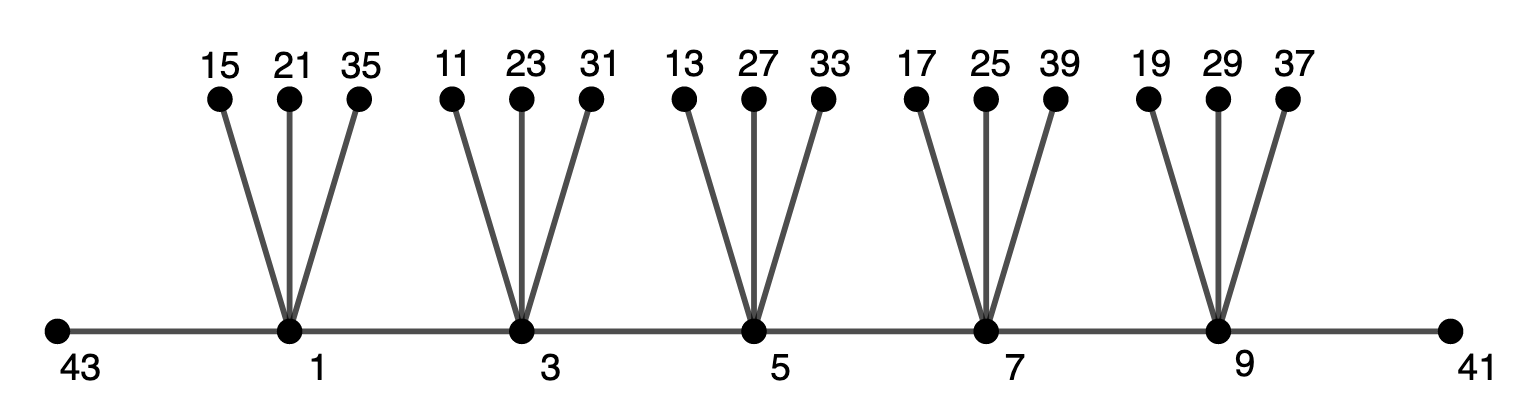}
\caption{Odd prime labelings of two $t$-toed caterpillars, one of length $10$ and $t=1$ where a reordering of the spine labels is needed, and one with length $7$ and $t=3$.}\label{caterpillar2} 
\end{center}
\end{figure}

This particular caterpillar was also shown in~\cite{Tout} to have a prime labeling for any length and any number of toes, where a coprime matching theorem developed in~\cite{Po_Se} was utilized to pair labels on the spine to the ones on the toes.   We use a similar approach relying on a result by Robertson and Small, particularly Theorem 3 in~\cite{Robertson_Small}.  They proved that given integers $a$ and $b$ which share no odd prime divisors less than $2n-1$, there exists a mapping $h$ from the set of the first $n$ odd integers, $O=\{1+2r \;|\; 0\leq r\leq n-1\}$, to the set $S=\{a+rb\;|\; 0\leq r\leq n-1\}$ in which $\gcd(m,h(m))=1$ for all $m\in O$.  We will also rely on the following reordering lemma to aid in this proof.


\begin{lemma}\label{reordering}
Given a sequence of $n+1$ consecutive odd integers $\{x_i\}_{i=1}^n = \{m, m+2,\ldots, m+2n\}$ and a specified entry $x_j$, there exists a reordering into the sequence $\{x_{i_1}, x_{i_2},\ldots, x_{i_n}, x_j\}$ such that each successive difference is a power of $2$.  That is, for each $k=2,3,\ldots, n$, $|x_{i_k}-x_{i_{k-1}}|=2^{s}$ and $|x_{j}-x_{i_n}|=2^t$ for some integers $s,t\geq 1$.
\end{lemma}

\begin{proof}
If $j=n$, then the original sequence does not need to be reordered since every difference is~$2$ initially.  For $j<n$, we consider the binary expansion of the difference of the last term and our specified entry: $m+2n-x_j=2^{a_1}+2^{a_2}+\cdots+2^{a_b}$.  We can assume $a_1>a_2>\cdots >a_b\geq 1$ where the last inequality is true since $m+2n$ and $x_j$ are both odd.  

We first alter the sequence by removing the entries $x_j, x_j+2^{a_1}, x_j+2^{a_1}+2^{a_2}, \ldots, x_j+2^{a_1}+\cdots+2^{a_{b-1}}$ and will refer to the remaining entries as $x_{i_1},x_{i_2},\ldots, x_{i_k},m+2n$.  Observe that since the exponents in the binary expansion were each greater than 1, no pair of consecutive entries were removed from the sequence.  Therefore, the successive differences of the shortened sequence are either 2 or 4.

We now append our removed entries at the end, but in reverse order, to create the following sequence: $$x_{i_1},x_{i_2},\ldots, x_{i_k},m+2n, x_j+2^{a_1}+\cdots+2^{a_{b-1}}, \ldots, x_j+2^{a_1}+2^{a_2}, x_j+2^{a_1}, x_j.$$
The pair $m+2n$ and $x_j+2^{a_1}+\cdots+2^{a_{b-1}}$ has a difference of $2^{a_b}$, and the subsequent pairs differ by each successive power of $2$ from the binary expansion. Thus, we have reordered the sequence with $x_j$ as the last term with every successive pair differing by a power of $2$.
\end{proof}

\begin{theorem}\label{t-toed caterpillar}
All $t$-toed caterpillars are odd prime for any length and any number of toes $t$.
\end{theorem}
\begin{proof}
Consider a caterpillar with length $n+2$ in which the $n$ interior vertices are of degree $t+2$.  Then there are $n(t+1)+2$ vertices, giving us a labeling set of $\{1,3,\ldots, 2n(t+1)+1, 2n(t+1)+3\}$.  We refer to the vertices on the spine as $u, v_1,v_2,\cdots, v_n, w$ and the toes from interior vertex $v_i$ as $x_{i,1}, x_{i,2}, \ldots, x_{i,t}$.

We begin our labeling $\ell$ by assigning $\ell(v_i)=2i-1$ for each $i=1,2,\ldots, n$.  Our two largest labels, $2n(t+1)+1$ and $2n(t+1)+3$, will be used to label the ends of the spine, vertices $u$ and $w$.  However, it is possible that both of these labels share a common factor with $\ell(v_n)=2n-1$.  Additionally, if $t$ is sufficiently large compared to $n$, one of the two labels may share a common factor with each label from $3$ to $2n-1$.  Even in this most extreme case, the other of $2n(t+1)+1$ or $2n(t+1)+3$ must be relatively prime with $3$ since these two largest labels differ by 2. Hence, we can conclude that one of $2n(t+1)+1$ and $2n(t+1)+3$ will be relatively prime with at least one of the interior labels $3,5,\ldots, 2n-1$, say the label $2k-1$ for some $2\leq k\leq n$.  Then we relabel $\ell(v_n)=2k-1$, assign as $\ell(w)$ whichever of $2n(t+1)+1$ and $2n(t+1)+3$ is relatively prime with $2k-1$, and assign the other of those two labels to $u$. By Lemma \ref{reordering}, we can reorder the labels $3,5,\ldots, 2n-1$ on the interior vertices to maintain $\ell(v_1)=1$, make this reassignment of $\ell(v_n)=2k-1$, and have the labels of each pair of adjacent vertices differ by a power of $2$.  Note that the vertex $u$ is only adjacent to $v_1$, which is labelled by $1$, so our current labeling of the spine has relatively prime labels at each adjacent pair of vertices.

For each $j=1, 2,\ldots, t$, we next assign the labels $2nj+1,2nj+3,\ldots, 2n(j+1)-1$ to the vertices $x_{1,j},x_{2,j},\ldots, x_{n,j}$.  By Theorem 3 in~\cite{Robertson_Small}, there is a function $h$ pairing the set of labels of the interior vertices $\{1,3,\ldots 2n-1\}$ with the set $\{2nj+1,2nj+3,\ldots, 2n(j+1)-1\}$ such that $\gcd(m,h(m))=1$ for all $m\in \{1,3,\ldots 2n-1\}$.  Notice that we are applying this theorem with $a=2nj+1$ and $b=2$, so clearly the condition about odd prime divisors of these two integers is met.  Thus, for a vertex $v_i$ with $\ell(v_i)=m$, we assign $\ell(x_{i,j})=h(m)$, resulting in this interior vertex and its adjacent leaf having relatively prime labels.  Continuing this for each $j$ up to the value $t$ will label all the toes of the caterpillar, ending with the label $2n(t+1)-1$.  Thus, we have used the labels $1,3,\ldots, 2n(t+1)-1, 2n(t+1)+1, 2n(t+1)+3$, and every pair of adjacent vertices is labelled with relatively prime integers, proving our labeling is odd prime.
\end{proof}

The last tree we will examine is a firecracker graph, denoted $F_{n,k}$.  This graph consists of a path with $n$ vertices, in which each vertex in the path is joined by an edge with the center of a $(k-1)$-star.  Firecrackers were shown in~\cite{Robertson_Small} to be prime.  See Figure~\ref{firecracker} for an example of a firecracker graph with an odd prime labeling.

\begin{figure}[t]
\begin{center}
\includegraphics[width=10 cm]{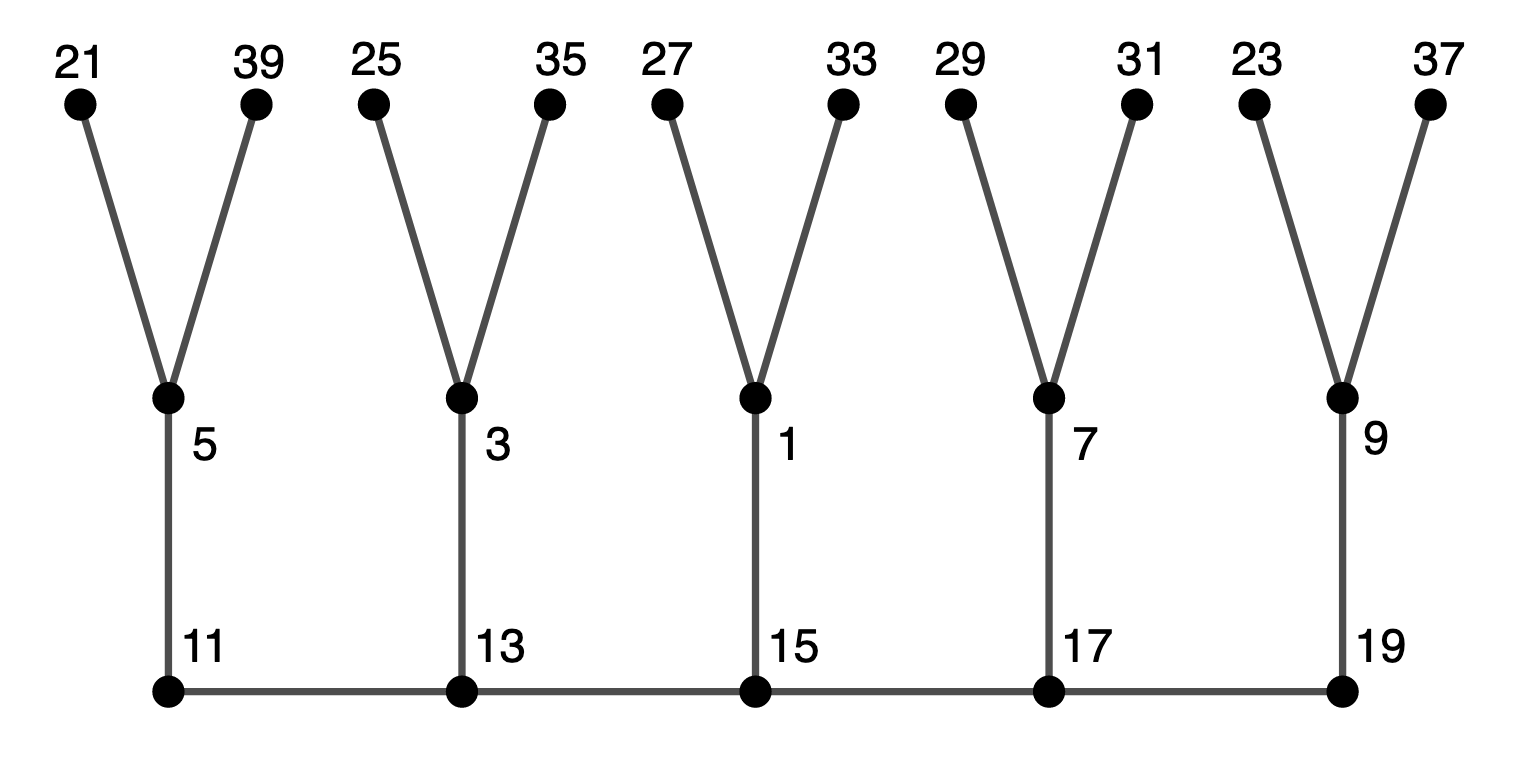}
\caption{An odd prime labeling of a firecracker graph $F_{5,4}$}\label{firecracker} 
\end{center}
\end{figure}

\begin{theorem}
The firecracker graph $F_{n,k}$ is odd prime for all $n>1$ and $k\geq 3$.
\end{theorem}

\begin{proof}
Let $v_1,v_2,\ldots, v_n$ be the vertices on the path, $u_1,u_2,\ldots, u_n$ be the centers of the adjoined stars, and $w_{i,1},w_{i,2},\ldots, w_{i,k-2}$ be the leaves from each vertex $u_i$.  We start by labeling the path vertices $v_1$ to $v_n$ with the odd integers from $2n+1$ to $4n-1$, respectively.  We then apply the coprime matching theorem used in Theorem~\ref{t-toed caterpillar} to create the bijection $h: \{1,3,\ldots, 2n-1\}\rightarrow \{2n+1,2n+3,\ldots, 4n-1\}$.  Now use the inverse of $h$ to assign the label for the center of each star as $\ell(u_i)=h^{-1}(\ell(v_i)).$  We next apply the same matching theorem to create the labels for $w_{i,j}$ from $j=1,2,\ldots,k-2$ by using $\ell(w_{i,j})=h_j(\ell(u_i))$ where the function $h_j$ is mapped from $\{1,3,\ldots, 2n-1\}$ to $\{2(j+1)n+1,2(j+1)n+3,\ldots, 2(j+2)n-1\}$. 

The adjacent pairs $v_iv_{i+1}$ have relatively prime labels since they are labelled by consecutive odd integers.  For the remaining adjacent vertices, we have $\gcd(\ell(u_i),\ell(v_i))=1$ and $\gcd(\ell(u_i),\ell(w_{i,j}))=1$ because the inputs and outputs of $h$ and $h_j$ are relatively prime by Theorem 3 in~\cite{Robertson_Small}.  Thus, every firecracker $F_{n,k}$ has an odd prime labeling.
\end{proof}

\section{Powers of graphs}\label{powers}

All of the graphs discussed thus far have been odd prime or at least conjectured to be so.  However, we conclude our investigation of specific classes of graphs with a class involving powers of graphs that is not odd prime for certain cases.  In general we define the $k$th power of a graph $G$, denoted by $G^k$, to have the same vertex set as $G$, but additional edges are included between vertices $u,v\in V(G)$ for which $d(u,v)\leq k$.  By $d(u,v)$, we are referring to the distance between the two vertices, or likewise the length of the shortest path from $u$ to $v$.  

We particularly will examine powers of paths and cycles.  Seoud and Youssef~\cite{S_Y} showed that $P_n^k$ and $C_n^k$ for $k\geq 2$ were not prime with exception of a few small cases of $n$, such as when $n=3$ since these graphs would just be $K_3$.  We will show that $P_n^2$, the square of a path, is odd prime for all $n$; however, the square of a cycle and higher powers of these two graphs won't always be odd prime.  An odd prime labeling of $P_{11}^2$ can be observed in Figure~\ref{path squared}, where the path consists of the diagonal edges while the connections between vertices of distance $2$ are horizontal.

\begin{figure}[t]
\begin{center}
\includegraphics[width=10 cm]{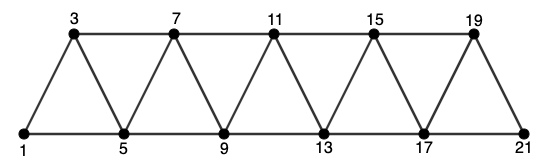}
\caption{An odd prime labeling of $P_{11}^2$}\label{path squared} 
\end{center}
\end{figure}

\begin{theorem}\label{pathsquared}
The square of a path, $P_n^2$, is odd prime for any length $n$. 
\end{theorem}

\begin{proof}
Consider the vertices of the path $P_n$ as $v_1,v_2,\ldots, v_n$ with the additional edges of the form $v_iv_{i+2}$ from $i=1$ to $n-2$ to form $P_n^2$.  We will label the vertices as $\ell(v_i)=2i-1$ for all $1\leq i \leq n$.  Edges in~$P_n$ clearly have relatively prime labels at the endpoints since they are labeled with consecutive odd integers.  Likewise, for edges of the form $v_iv_{i+2}$,
$\gcd(2i-1,2(i+2)-1)=\gcd(2i-1,2i+3)=1$ since the difference of the labels is 4.  Thus, this is an odd prime labeling.
\end{proof}

We were able to successfully label $P_n^2$ with an odd prime labeling because its independence number is $\beta(P_n^2)=\left\lceil \frac{n}{3}\right\rceil$, within the required range given in Lemma~\ref{ind req} to be able to place the multiples of 3 on an independent set of vertices.  However, as discussed in~\cite{S_Y} and~\cite{A_F}, the cycle squared has an independence number of $\beta(C_n^2)=\left\lfloor \frac{n}{3}\right\rfloor$, and for powers of $k\geq 3$, the independence numbers generalize to be the following:
$$\beta(P_n^k)=\left\lceil \frac{n}{k+1}\right\rceil\text{ and }\beta(C_n^k)=\left\lfloor \frac{n}{k+1}\right\rfloor.$$
This will ultimately limit our abilities to develop odd prime labelings of certain cases of $C_n^2$. Nonetheless, specific cases such as $C_7^2$, as seen in Figure~\ref{cycle squared}, can be created, and on the higher powers of certain paths and cycles as seen in the following results.

\begin{figure}[t]
\begin{center}
\includegraphics[width=5 cm]{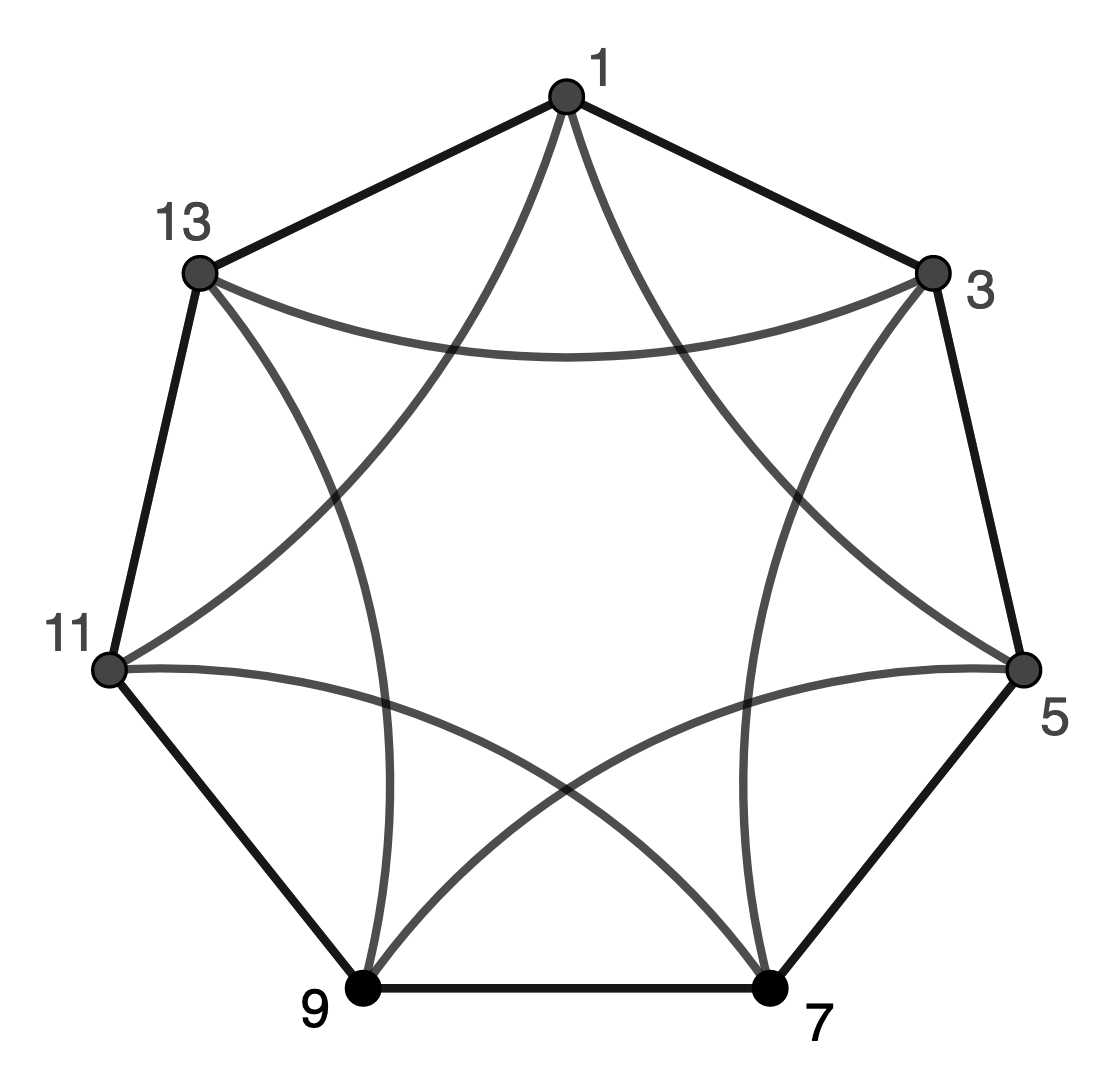}
\caption{An odd prime labeling of the squared cycle $C_7^2$}\label{cycle squared}
\end{center}
\end{figure}

\begin{theorem}
The square of a cycle, $C_n^2$, is odd prime if and only if $n\not\equiv 2\pmod{3}$.
\end{theorem}
\begin{proof}
For the case of $n\equiv 2\pmod{3}$, we have $n=3k+2$ for some $k\in \mathbb{Z}$.  Then the independence number of $C_n^2$ is 
$$\beta(C_n^2)=\left\lfloor \frac{n}{3}\right\rfloor=\left\lfloor \frac{3k+2}{3}\right\rfloor=k.$$
However, Lemma~\ref{ind req} requires $\beta(G)$ to be at least $\lfloor \frac{n+1}{3}\rfloor=\lfloor \frac{3k+3}{3}\rfloor=k+1$, and hence $C_n^2$ is not odd prime in this case.

For when $n\not\equiv 2\pmod{3}$, we label the vertices of the cycle $v_1,v_2,\ldots, v_n$ as we did for $P_n^2$ in Theorem~\ref{pathsquared} with $\ell(v_i)=2i-1$.  Since this was an odd prime labeling of the square of the path, we only to check the additional edges $v_{n-1}v_1$, $v_nv_1$, and $v_nv_2$.  We have that $\gcd(2n-3,1)$, $\gcd(2n-1,1)$, and $\gcd(2n-1,3)$ are each 1, with the last one being due to our assumption of $n\not\equiv 2\pmod{3}$ forcing $2n-1$ to not be a multiple of 3.  Thus, $C_n^2$ is odd prime in these two cases of $n\equiv 0$ or $1\pmod{3}$.
\end{proof}

When examining higher powers of paths and cycles, the trivial labelings used in the previous two theorems do not work.  However, some small cases for $n$ will result in $P_n^k$ being odd prime once labels are reordered as long as the independence number is large enough to create an odd prime labeling.  One such example is $P_{13}^3$ as shown in Figure~\ref{path cubed}.  The following theorem fully characterizes when $P_n^k$ and $C_n^k$ are odd prime for $k\geq 3$, where we note that only $n\geq k+2$ for paths are considered since the graph would otherwise simply be the complete graph $K_n$.

\begin{figure}[t]
\begin{center}
\includegraphics[width=15 cm]{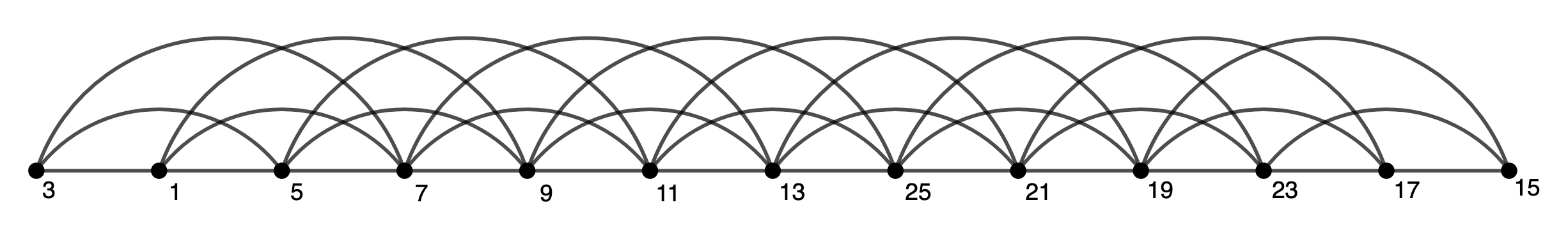}
\caption{An odd prime labeling of $P_{13}^3$}\label{path cubed} 
\end{center}
\end{figure}

\begin{theorem}
The following are true regarding the graphs $P_n^k$ and $C_n^k$
\begin{itemize}
    \item $P_n^3$ is odd prime if and only if $n=5,6,7,9,10,$ or $13$.
    \item $P_n^4$ is odd prime if and only if $n=6$ or $7$.
    \item $P_n^5$ is odd prime if and only if $n=7$.
    \item $P_n^k$ with $k\geq 6$ is not odd prime for any $n\geq k+2$.
    \item $C_n^k$ with $k\geq 3$ is not odd prime for any $n\geq k+2$.
\end{itemize}
\end{theorem}
\begin{proof}
Since we have $\displaystyle\beta(P_n^k)=\left\lceil \frac{n}{k+1}\right\rceil$ as the independence number of $P_n^k$, the inequality $$\beta(P_n^3)=\left\lceil \frac{n}{4}\right\rceil<\left\lfloor \frac{n+1}{3}\right\rfloor$$ 
can be verified for $n=8, 11, 12$, or when $n\geq 14$. Similarly, 
$$\beta(P_n^4)=\left\lceil \frac{n}{4}\right\rceil<\left\lfloor \frac{n+1}{3}\right\rfloor$$ 
for all $n\geq 8$ and 
$$\beta(P_n^5)=\left\lceil \frac{n}{6}\right\rceil<\left\lfloor \frac{n+1}{3}\right\rfloor$$ 
for all $n\geq 8$ are true as well.  Finally, we have $\displaystyle\beta(C_n^k)=\left\lfloor \frac{n}{k+1}\right\rfloor<\left\lfloor \frac{n+1}{3}\right\rfloor$ for all $n\geq 5$ and $k\geq 3$.  By Lemma~\ref{ind req}, these inequalities justify each of the non-odd prime cases in the statement.

For the cases where these graphs are odd prime, we use the following sequences to label the vertices of the path:
\begin{multicols}{2}
\begin{itemize}
    \item $P_5^3$: $3,1,5,7,9$
    \item $P_6^3$: $3,1,5,7,9,11$
    \item $P_7^3$: $3,1,5,7,9,11,13$
    \item $P_9^3$: $3,1,5,7,9,11,13,17,15$
    \item $P_{10}^3$: $3,1,5,7,9,11,13,17,15,19$
    \item $P_{13}^3$: $3,1,5,7,9,11,13,25,21,19,23,17,15$
    \item $P_6^4$: $3,1,5,7,11,9$
    \item $P_7^4$: $3,1,5,7,11,9,13$
    \item $P_7^5$: $3,1,5,7,11,13,9$.
\end{itemize}
\end{multicols}
One can see that any labels that share factors of $3$, $5$, or $7$ are at least distance $k+1$ apart on the path, making them odd prime labelings.
\end{proof}

\section{Prime vs. Odd Prime Graphs}\label{pvop}

Results such as Theorems~\ref{unioncycles} (when considering odd-length cycles), \ref{prismgraphs} (also for odd-length prisms), \ref{triangular}, \ref{pentagonal}, and~\ref{pathsquared} show that many classes of graphs are not prime, but have an odd prime labeling.  On the other hand, no examples of graphs have been found to be prime, but not have an odd prime labeling.  Our expectation is that no such graphs exist, as summarized for the following conjecture, first introduced as Conjecture 2.14 in~\cite{P_S} and Conjecture $1$ in~\cite{Y_A}.  

\begin{conjecture}\label{prime-to-odd-prime}
Every prime graph is an odd prime.
\end{conjecture}

One potential avenue to prove this conjecture involves the maximal prime graph, $R_n$, defined in~\cite{S_Y} as the graph consisting of vertices $v_1,v_2,\ldots, v_n$ where $v_iv_j$ is an edge if and only if $\gcd(i,j)=1$.
Since all prime graphs of order $n$ are isomorphic to a spanning subgraph of $R_n$, Conjecture~\ref{prime-to-odd-prime} would be proven true if one could develop an odd prime labeling of $R_n$ for any $n$.  Namely, we need a function $\ell: \{v_1,v_2,\ldots, v_n\} \rightarrow \{1, 3, \ldots, 2n-1\}$ for any $n\geq 1$ such that for all $i,j\in \{1,2,\ldots, n\}$, if $\gcd(i,j)=1$, then $\gcd(\ell(v_i),\ell(v_j))=1$.  This function could then be used convert a prime labeling of any prime graph to an odd prime labeling.  

\begin{figure}[t]
\begin{center}
\includegraphics[width=7.5 cm]{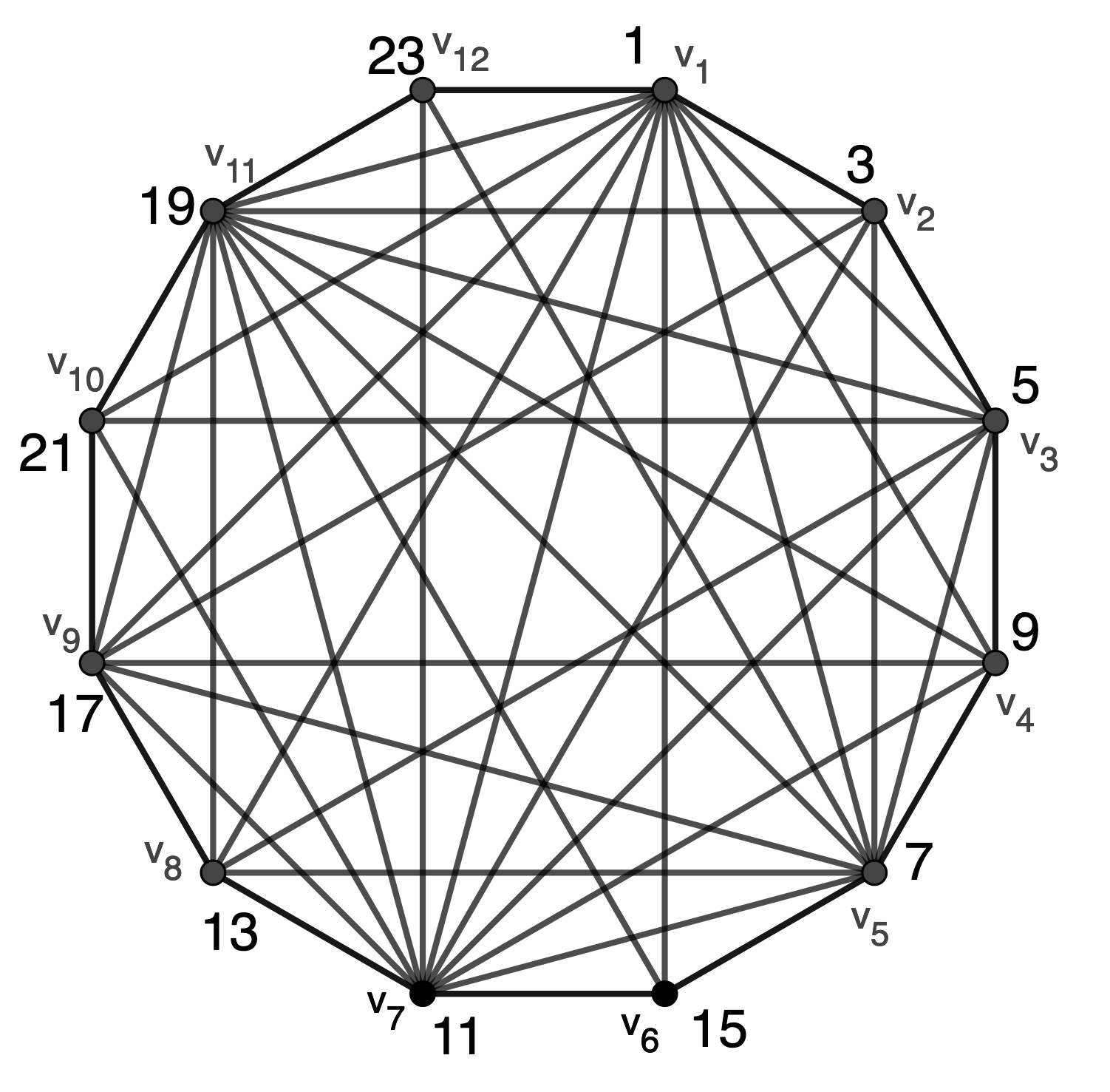}
\caption{An odd prime labeling of $R_{12}$}\label{maxprime} 
\end{center}
\end{figure}

We have been able to find general function for all $n$ still, but we conclude this paper by showing the progress that has been made in developing one.  In order to maintain the relatively prime condition on adjacent labels, our mapping is based on shifting each of the prime divisors of an input to the next largest prime.  As an example, for $v_{12}$ where $12=2\cdot2\cdot 3$, we assign $\ell(v_{20})=3\cdot3\cdot 5=45$, provided that $n$ is largest enough for $2n-1\geq 45$.  When $n$ is not large enough for $45$ to be a usable label, one of the available prime numbers between $n$ and $2n-1$ is usually assigned.  An example of one of these labeling functions being applied to the maximal prime graph $R_{12}$ is shown in Figure~\ref{maxprime}.

\begin{theorem}
All prime graphs of order $n\leq 50$ are odd prime.
\end{theorem}
\begin{proof}
To prove this, we introduce a set of functions that would assign an odd prime labeling to the vertices of the graph $R_n$ with $n\leq 50$.  Given a particular order $n$, Table~\ref{table} provides the sequence of labels used on the vertices $v_1,v_2,\ldots, v_n$ of $R_n$. Use the row given by the set in the first column which contains the given $n$, and then the first $n$ values in that row (or two rows) will be an odd prime labeling $\ell$ of $R_n$.  A Mathematica program has been used to verify the relatively prime condition in which for each $i,j\in\{1,2,\ldots, n\}$ with $\gcd(i,j)=1$, we have $\gcd(\ell(i),\ell(j))=1$ as well.  
\end{proof}

\begin{table}
\addtolength{\tabcolsep}{-2pt}
\begin{tabular}{|c | c | c | c | c | c | c | c | c | c | c | c | c | c | c | c | c | c | c | c | c | c | c | c | c | c | }
 \hline
 $n$ & \multicolumn{24}{l}{Sequence for $\ell(v_1),\ell(v_2),\ldots, \ell(v_n)$} & \\ \hline
$[1,4]$ & 1 & 3 & 5 & 7 & \multicolumn{20}{l}{} & \\ \hline
$[5,7]$ & 1 & 3 & 5 & 9 & 7 & 11 & 13 & \multicolumn{17}{l}{} & \\ \hline
$[8,10]$ & 1 & 3 & 5 & 9 & 7 & 15 & 11 & 13 & 17 & 19 & \multicolumn{14}{l}{} & \\ \hline
$[11,12]$ & 1 & 3 & 5 & 9 & 7 & 15 & 11 & 13 & 17 & 21 & 19 & 23 &  \multicolumn{12}{l}{} & \\ \hline
$[13,16]$ & 1 & 3 & 5 & 9 & 7 & 15 & 11 & 13 & 25 & 21 & 17 & 23 & 19 & 27 & 29 & 31 &  \multicolumn{8}{l}{} & \\ \hline
$\{17\}$ & 1 & 3 & 5 & 9 & 7 & 15 & 11 & 27 & 25 & 21 & 13 & 17 & 19 & 33 & 23 & 29 & 31 &  \multicolumn{7}{l}{} & \\ \hline
$[18,22]$ & 1 & 3 & 5 & 9 & 7 & 15 & 11 & 27 & 25 & 21 & 17 & 13 & 19 & 33 & 35 & 23 & 29 & 31 & 37 & 39 & 41 & 43&  \multicolumn{2}{l}{} & \\ \hline
\multirow{2}{*}{$[23, 27]$} & 1 & 3 & 5 & 9 & 7 & 15 & 11 & 27 & 25 & 21 & 13 & 45 & 17 & 33 & 35 & 19 & 23 & 29 & 31 & 37 & 41 & 39 & 43 & 47 & 49 \\ \cline{2-26}
 & 51 & 53& \multicolumn{22}{l}{} & \\ \hline
\multirow{2}{*}{$\{28\}$} & 1 & 3 & 5 & 9 & 7 & 15 & 11 & 27 & 25 & 21 & 13 & 45 & 17 & 33 & 35 & 19 & 23 & 29 & 31 & 41 & 55 & 39 & 37 & 43 & 49 \\ \cline{2-26}
 & 51 & 47 & 53& \multicolumn{21}{l}{} & \\ \hline
\multirow{2}{*}{$[29, 31]$} & 1 & 3 & 5 & 9 & 7 & 15 & 11 & 27 & 25 & 21 & 13 & 45 & 17 & 33 & 35 & 57 & 23 & 19 & 29 & 37 & 55 & 39 & 31 & 41 & 49 \\ \cline{2-26}
 & 51 & 43 & 47 & 53 & 59 & 61& \multicolumn{18}{l}{} & \\ \hline
\multirow{2}{*}{$[32, 33]$} & 1 & 3 & 5 & 9 & 7 & 15 & 11 & 27 & 25 & 21 & 13 & 45 & 17 & 33 & 35 & 57 & 23 & 19 & 29 & 63 & 55 & 39 & 31 & 37 & 49 \\ \cline{2-26}
 & 51 & 41 & 43 & 47 & 53 & 59 & 61 & 65& \multicolumn{16}{l}{} & \\ \hline
\multirow{2}{*}{$\{34\}$} & 1 & 3 & 5 & 9 & 7 & 15 & 11 & 27 & 25 & 21 & 13 & 45 & 17 & 33 & 35 & 23 & 19 & 29 & 31 & 63 & 55 & 39 & 37 & 41 & 49 \\ \cline{2-26}
 & 51 & 43 & 47 & 53 & 59 & 61 & 67 & 65 & 57& \multicolumn{15}{l}{} & \\ \hline
\multirow{2}{*}{$[35, 37]$} & 1 & 3 & 5 & 9 & 7 & 15 & 11 & 27 & 25 & 21 & 13 & 45 & 17 & 33 & 35 & 23 & 19 & 29 & 31 & 63 & 55 & 39 & 37 & 41 & 49 \\ \cline{2-26}
& 51 & 43 & 47 & 53 & 59 & 61 & 69 & 65 & 57 & 67 & 71 & 73& \multicolumn{12}{l}{} & \\ \hline
\multirow{2}{*}{$\{38\}$} & 1 & 3 & 5 & 9 & 7 & 15 & 11 & 27 & 25 & 21 & 13 & 45 & 17 & 33 & 35 & 23 & 19 & 75 & 29 & 63 & 55 & 39 & 31 & 37 & 49 \\ \cline{2-26}
& 51 & 41 & 43 & 47 & 53 & 59 & 69 & 65 & 57 & 61 & 67 & 71 & 73& \multicolumn{11}{l}{} & \\ \hline
\multirow{2}{*}{$[39, 40]$} & 1 & 3 & 5 & 9 & 7 & 15 & 11 & 27 & 25 & 21 & 13 & 45 & 17 & 33 & 35 & 23 & 19 & 75 & 29 & 63 & 55 & 39 & 31 & 37 & 49 \\ \cline{2-26}
& 51 & 41 & 43 & 47 & 53 & 59 & 69 & 65 & 57 & 77 & 61 & 67 & 71 & 73 & 79& \multicolumn{9}{l}{} & \\ \hline
\multirow{2}{*}{$[41, 42]$} & 1 & 3 & 5 & 9 & 7 & 15 & 11 & 27 & 25 & 21 & 13 & 45 & 17 & 33 & 35 & 81 & 19 & 75 & 23 & 63 & 55 & 39 & 29 & 31 & 49 \\ \cline{2-26}
& 51 & 37 & 41 & 43 & 47 & 53 & 59 & 65 & 57 & 77 & 61 & 67 & 69 & 71 & 73 & 79 & 83& \multicolumn{7}{l}{} & \\ \hline
\multirow{2}{*}{$\{43\}$} & 1 & 3 & 5 & 9 & 7 & 15 & 11 & 27 & 25 & 21 & 13 & 45 & 17 & 33 & 35 & 81 & 19 & 75 & 23 & 63 & 55 & 39 & 29 & 31 & 49 \\ \cline{2-26}
& 51 & 37 & 41 & 43 & 47 & 53 & 59 & 65 & 57 & 77 & 61 & 67 & 69 & 85 & 71 & 73 & 79 & 83& \multicolumn{6}{l}{} & \\ \hline
\multirow{2}{*}{$[44, 45]$} & 1 & 3 & 5 & 9 & 7 & 15 & 11 & 27 & 25 & 21 & 13 & 45 & 17 & 33 & 35 & 81 & 19 & 75 & 23 & 63 & 55 & 39 & 31 & 29 & 49 \\ \cline{2-26}
& 51 & 37 & 41 & 43 & 47 & 53 & 59 & 65 & 57 & 77 & 61 & 67 & 69 & 85 & 71 & 73 & 79 & 83 & 87 & 89& \multicolumn{4}{l}{} & \\ \hline
\multirow{2}{*}{$\{46\}$} & 1 & 3 & 5 & 9 & 7 & 15 & 11 & 27 & 25 & 21 & 29 & 45 & 17 & 33 & 35 & 81 & 19 & 75 & 23 & 63 & 55 & 39 & 31 & 65 & 49 \\ \cline{2-26}
& 51 & 37 & 13 & 41 & 91 & 43 & 47 & 53 & 57 & 77 & 59 & 61 & 69 & 85 & 67 & 71 & 73 & 79 & 87 & 83 & 89& \multicolumn{3}{l}{} & \\ \hline
\multirow{2}{*}{$\{47\}$} & 1 & 3 & 5 & 9 & 7 & 15 & 11 & 27 & 25 & 21 & 29 & 45 & 17 & 33 & 35 & 81 & 19 & 75 & 23 & 63 & 55 & 39 & 31 & 65 & 49 \\ \cline{2-26}
& 51 & 37 & 13 & 41 & 91 & 43 & 47 & 53 & 57 & 77 & 59 & 61 & 69 & 85 & 67 & 71 & 73 & 79 & 87 & 83 & 93 & 89& \multicolumn{2}{l}{} & \\ \hline
\multirow{2}{*}{$[48, 49]$} & 1 & 3 & 5 & 9 & 7 & 15 & 11 & 27 & 25 & 21 & 29 & 45 & 17 & 33 & 35 & 81 & 59 & 75 & 23 & 63 & 55 & 39 & 31 & 65 & 49 \\ \cline{2-26}
& 51 & 37 & 13 & 41 & 91 & 43 & 47 & 19 & 53 & 77 & 57 & 61 & 69 & 85 & 67 & 71 & 73 & 79 & 87 & 83 & 93 & 89 & 95 & 97&   \\ \hline
\multirow{2}{*}{$\{50\}$} & 1 & 3 & 5 & 9 & 7 & 15 & 11 & 27 & 25 & 21 & 29 & 45 & 17 & 33 & 35 & 81 & 59 & 75 & 23 & 63 & 55 & 39 & 31 & 65 & 49 \\ \cline{2-26}
 & 51 & 37 & 99 & 41 & 91 & 43 & 47 & 19 & 53 & 77 & 57 & 61 & 69 & 85 & 67 & 71 & 73 & 79 & 87 & 83 & 93 & 89 & 95 & 97 & 13 \\ \hline
\end{tabular}
\caption{A table containing the sequence of labels used for an odd prime labeling of $R_n$ for $n\leq 50$}
\label{table}
\end{table}

\section*{Acknowledgements}
The authors are grateful for the support of the Austin Peay State University Department of Mathematics and Statistics.  They also thank the referees for their comments and suggestions.

\newcommand{\journal}[6]{{#1,} {#2}, {\it #3} {\bf #4} (#5) #6.}
\newcommand{\dissertation}[4]{{#1,} #2, {\it #3,} #4}
\newcommand{\book}[5]{{#1,} {\it #2,} #3, #4.}


\begin{thebibliography}{9}
      
     
      
\bibitem{A_F}
\journal{J.\ Asplund and N.\ B.\ Fox}
    {Minimum coprime labelings for operations on graphs}
    {Integers}
    {19}{2019}{Article A24}
      
\bibitem{A_F2}      
\journal{J.\ Asplund and N.\ B.\ Fox}
    {Minimum coprime labelings of generalized Petersen and prism graphs}
    {J.\ Integer Seq.}
    {24}{2021}{Article 21.3.5}
    
\bibitem{BDHMMM}
\journal{A.\ Berliner, N.\ Dean, J.\ Hook, A.\ Marr, A.\ Mbirika, and C.\ McBee}
    {Coprime and prime labelings of graphs}
    {J.\ Integer Seq.}
    {19}{2016}{Article 16.5.8}
    
\bibitem{C_F}
\journal{M.\ Cloys and N.\ B.\ Fox}
    {Neighborhood-prime labelings of trees and other classes of graphs}
    {Pi Mu Epsilon Journal}
    {15}{2019}{9--21}
      
\bibitem{Dean}
\journal{N.\ Dean}
    {Proof of the prime ladder conjecture}
    {Integers}
    {17}{2017}{Article A40}
      
\bibitem{D_L_M}
\journal{T.\ Deretsky, S.\ M.\ Lee, and J.\ Mitchem}
    {On vertex prime labelings of graphs}
    {Graph Theory, Combinatorics, and Applications Vol. 1}
    {}{1991}{359--369}
    
    
\bibitem{D_E}
\journal{N.\ Diefenderfer, D.\ Ernst, M.\ Hastings, L.\ Heath, H.\ Prawzinky, B.\ Preston, J.\ Rushall, E.\ White, and A.\ Whittemore}
    {Prime Vertex Labelings of Several Families of Graphs}
    {Involve}
    {9}{2016}{No.\ 4, 667--688}
      
\bibitem{F_H}
\journal{H.\ L.\ Fu and K.\ C.\ Huang}
	{On prime labellings}
	{Discrete Math.}
	{127}{1994}{181--186}     
	
\bibitem{Gallian}
\journal{J.\ A.\ Gallian}
	{A dynamic survey of graph labeling}
	{Electron.\ J.\ Combin.}
	{DS6}{2014}{\hspace{-.14cm}} 
	
\bibitem{HLYZ}
\journal{K.\ Haque, X.\ Lin, Y.\ Yang, and P.\ Zhao}
    {On the prime labeling of generalized Petersen graph $P(n,1)$}
    {Util.\ Math.}
    {83}{2010}{95--106}
	
\bibitem{Lee}
\journal{C.\ Lee}
    {Minimum coprime graph labelings}
    {J.\ Integer Seq.}
    {23}{2020}{Article 20.11.4}
    
\bibitem{V_S_N}
\journal{T.\ Nicholas, S.\ Somasundaram, and V.\ Vilfred}
    {Classes of prime labelled graphs}
    {International Journal of Management and Systems}
    {18}{2002}{no. 2}    
	
\bibitem{P_S}
\journal{U.\ M.\ Prajapati and K.\ P.\ Shah}
    {On odd prime labeling}
    {Int.\ J.\ Res.\ Anal.\ Rev.}
    {5}{2018}{284--294}

\bibitem{Pa_Sh}
\journal{S.\ K.\ Patel and N.\ P.\ Shrimali}
    {Neighborhood-prime labeling}
    {Int.\ J.\ Math.\ Soft Comput.}
    {6}{2015}{135--143}
    
\bibitem{Po_Se}
\journal{C.\ Pomerance and J.\ L.\ Selfridge}
    {Proof of D.\ J.\ Newman's coprime mapping conjecture}
    {Mathematika}
    {27}{1980}{69--83}

\bibitem{Robertson_Small}
\journal{L.\ Robertson and B.\ Small}
	{On Newman's conjecture and prime trees}
	{Integers}
	{9}{2009}{117--128}
	
\bibitem{S_Y}
\journal{M.\ A.\ Seoud and M.\ Z.\ Youssef}
    {On prime labelings of graphs}
    {Congr.\ Numer.}
    {141}{1999}{203--215}
    
\bibitem{S_P_S}
\journal{M.\ Sundaram, R.\ Ponraj, and S.\ Somasundaram}
    {On a prime labeling conjecture}
    {Ars Combin.}
    {79}{2006}{205--209}
      
\bibitem{Tout}
\journal{A.\ Tout, A.\ N.\ Dabboucy, and K.\ Howalla}
      {Prime labeling of graphs}
      {Nat.\ Acad.\ Sci.\ Letters}
      {11}{1982}{365--368}      
      
\bibitem{V_P}
\journal{S.\ K.\ Vaidya and U.\ M.\ Prajapati}
    {Some results on prime and $k$-prime labeling}
    {J.\ Math.\ Research}
    {3}{2011}{66--75}
    
      
\bibitem{Y_A}
\journal{M.\ Z.\ Youssef and Z.\ S.\ Almoreed}
    {On odd prime labeling of graphs}
    {Open J.\ Discrete Appl.\ Math.}
    {3}{2020}{33--40}
    

\end{thebibliography}
\end{document}